\documentclass{scrartcl}
\usepackage[utf8]{inputenc} 
\usepackage{amsthm}
\usepackage{amsmath}
\usepackage{amssymb}
\usepackage{dsfont}
\theoremstyle{plain}
\newtheorem{Theorem}{Theorem}[section]
\newtheorem{Lemma}[Theorem]{Lemma}
\newtheorem{Corollary}[Theorem]{Corollary}
\newtheorem{Proposition}[Theorem]{Proposition}
\theoremstyle{definition}
\newtheorem{Definition}[Theorem]{Definition}
\newtheorem{Notation}[Theorem]{Notation}
\theoremstyle{remark}
\newtheorem{Remark}[Theorem]{Remark}
\newcounter{condition}
\renewcommand{\thecondition}{C\arabic{condition}}
\newenvironment{Condition}{\begin{trivlist}\refstepcounter{condition}\item[ \textbf{Condition}~\textbf{\thecondition.}]\itshape}{\end{trivlist}}
\newcounter{assumption}
\renewcommand{\theassumption}{A\arabic{assumption}}

\newcommand{\dd}[1]{{\operatorname{d}}#1}
\newcommand{\R}{\mathds{R}}
\newcommand{\Q}{\mathds{Q}}
\newcommand{\E}{\mathds{E}}
\newcommand{\N}{\mathds{N}}
\newcommand{\1}{\mathds{1}}

\newcommand{\Co}{\mathcal{C}}

\newcommand{\Prob}{\mathds{P}}
\newcommand{\norm}[1]{\left\lVert #1 \right\rVert}
\newcommand{\abs}[1]{\left| #1 \right|}

\DeclareMathOperator{\sgn}{sgn}

\title{On the Strong Feller Property for Stochastic Delay Differential Equations with Singular Drift}
\author{ Stefan Bachmann\\  {\small Email: bachmann@math.uni-leipzig.de} \\ \multicolumn{1}{p{.55\textwidth}}{\normalsize\centering\emph{Institut f\"ur Mathematik, Universit\"at Leipzig, Augustusplatz 10, 04109 Leipzig, Germany}}}
\begin{document}
	\maketitle
	
	\begin{abstract}
		\begin{center}
			\textbf{Abstract}
		\end{center}
		In this paper, we prove the strong Feller property for stochastic delay (or functional) differential equations with singular drift. We extend an approach of Maslowski and Seidler to derive the strong Feller property of those equations, see \cite{Maslowski2000}. The argumentation is based on the well-posedness and the strong Feller property of the equations' drift-free version. To this aim, we investigate a certain convergence of random variables in topological spaces in order to deal with discontinuous drift coefficients.
	\end{abstract}
	\section{Introduction}
	In this paper, we investigate a certain convergence of random variables in topological spaces and apply the results to prove an improved version of the strong Feller property of the following stochastic delay (or functional) differential equation (SDDE). 
	\begin{align}
		\begin{aligned}
			\label{eq}
			\dd{X^x}(t)&=B(t,X_t^x)\dd{t}+b(t,X^x(t))\dd{t}+\sigma(t,X^x(t))\dd{W}(t),\\
			X_0^x&=x
		\end{aligned}
	\end{align}
	where $W$ is a $d$-dimensional Brownian motion, $B:\R_{\geq0}\times C\left([-r,0],\R^d\right)\to\R^d$ is measurable and strictly sublinear in the second variable, $\sigma:\R_{\geq0}\times\R^d\to\R^{d\times d}$ is measurable, bounded, non-degenerate and Lipschitz in space and $b\in L_{loc}^q\left(\R_{\geq0};L^p\left(\R^d\right)\right)$ where $1<p,q$ fulfill
	\begin{align}
	\label{pq}
		\frac{d}{p}+\frac{2}{q}<1.
	\end{align}
	The well-posedness and stability of equation \eqref{eq} without a delay drift term, i.e. $B\equiv0$, has been studied thoroughly: Krylov and R\"ockner have shown existence and uniqueness for the special case $\sigma\equiv Id$ in \cite{KrylovRoeckner}, which had elaborated previous results, for example of Portenko \cite{Portenko}, Veretennikov \cite{Ver} and Zvonkin \cite{Zvonkin}. In \cite{RUTKOWSKI}, Rutkowski considered pathwise uniqueness for stochastic one-dimensional equations with singular drift involving local time conditions which were introduced by Barlow and Perkins. Mart\'inez and Gy\"ongy have proven well-posedness for non-constant $\sigma$ in \cite{Gyoengy2001}, however, under the stricter assumption $b\in L^{2d+2}(\R^{d+1})$. In \cite{Zhang2011}, Zhang was able to improve it further by only assuming $\abs{\nabla_x\sigma^{i,j}}\in L_{loc}^q\left(\R_{\geq0};L^p\left(\R^d\right)\right)$, $i,j=1,\dots, d$. Gawarecki and Mandrekar have studied SDEs with discontinuous drift and their connection to unbounded spin-systems in \cite{Gawarecki}. Blei and Engelbert have analyzed one-dimensional equations with so-called generalized drifts in \cite{BLEI} and haven given sufficient and necessary criteria for existence and uniqueness in distribution. A nonlinear Kolmogorov equation for stochastic functional delay differential equations with jumps has been proven by Cordoni, Di Persio and Oliva in \cite{Cordoni2017}. Cordoni and Di Persio have also shown existence and uniqueness of mild solutions for stochastic reaction-diffusion equations on networks with dynamic time-delayed boundary conditions in \cite{CORDONI2017583} and have considered an application to a stochastic optimal control problem. In \cite{Stefan1}, we extended the well-posedness result of Zhang \cite{Zhang2011} to the delay case, essentially by using a combination of Zvonkin's transformation, several Girsanov techniques and a stochastic Gronwall lemma from von Renesse and Scheutzow in \cite{Scheutzow,Renesse}. The existence of a unique strong solution for SDDEs with singular drift has also been shown in \cite{ChinesischerTyp}.
	
	Da Prato, Elworthy and Zabczyk have studied the strong Feller property for stochastic semilinear equations with unbounded coefficients in \cite{DaPrato}. The strong Feller property with respect to the state space $\R^d$ of equation \eqref{eq} with unbounded non-functional drift has been shown by Zhang in \cite{Zhang2011}. However, in this paper we are interested in the state space of path segments $C([-r,0],\R^d)$. Several Harnack inequalities have been studied for stochastic delay differential equations, which imply the strong Feller property. Es-Sarhir, von Renesse and Scheutzow have investigated the case $b\equiv0$ and $\sigma\equiv const$ in \cite{es-sarhir2009}. Wang and Yuan have established results for non-constant and uniformly non-degenerate diffusion coefficients, which do not depend on the past, in \cite{WangYuan}. By remark 1.4 in \cite{es-sarhir2009}, the strong Feller property might not be given if the diffusion term is of real functional nature. Both papers are based on a coupling technique.
	
	However, we do not use a coupling technique but the probabilistic approach of Maslow-\\
	ski and Seidler, see \cite{Maslowski2000}. In order to deal with discontinuous drift coefficients, we consider a certain convergence for topological spaces, see Theorem \ref{MainTheorem1}. As a result, we gain a simple method to derive the strong Feller property of equation \eqref{eq} from the well-posedness and the strong Feller property of the simpler special case $B\equiv 0$ and $b\equiv 0$.
	\begin{Notation}
		We denote by $\norm{\cdot}_{OP}$ and $\norm{\cdot}_{HS}$ the operator norm and respectively the Hilbert-Schmidt norm for matrices $A\in\R^{d\times d}$, i.e.
		$$\norm{A}_{op}=\sup\limits_{v\in\R^d,\abs{v}=1}\abs{Av}, \ \norm{A}_{HS}=\sqrt{\sum_{i,j=1}^{d}\abs{A^{i,j}}^2}.$$
		Additionally, we write for $a,b\in[-\infty,+\infty]$
		$$a\wedge b:=\min\{a,b\}, \ a\vee b:=\max\{a,b\}.$$
	\end{Notation}
	\begin{Notation}
		In the sequel, let $r>0$ be an arbitrary but fixed number and define
		$$\mathcal{C}:=C\left([-r,0],\R^d\right)$$
		equipped with the supremum norm $\norm{\cdot}_\infty$. For a process $X$ defined on $[t-r,t]$ with $t\geq0$, we write
		$$X_t(s):=X(t+s), \ s\in[-r,0].$$
		Furthermore, we introduce the following function spaces, which will be used later on: define for $0\leq S\leq T<\infty$ and $p,q\in(1,\infty)$
		$$
		\begin{array}{ll}
			L^q_p(S,T):=L^q\left([S,T];L^p\left(\R^d\right)\right), &L^q_p(T):=L^q_p(0,T),\\
			\mathds{H}^q_{2,p}(S,T):=L^q\left([S,T];W^{2,p}\left(\R^d\right)\right), &\mathds{H}^q_{2,p}(T):=\mathds{H}^q_{2,p}(0,T),\\
				H^q_{2,p}(S,T):=W^{1,q}\left([S,T];L^p\left(\R^d\right)\right)\cap\mathds{H}^q_{2,p}(,T), &H^q_{2,p}(T):=H^q_{2,p}(0,T),
		\end{array}
		$$
		equipped with the norm
		$$\norm{u}_{H^q_{2,p}(S,T)}:=\norm{\partial_tu}_{L^q_p(S,T)}+\norm{u}_{\mathds{H}^q_{2,p}(S,T)}, \ u\in H^q_{2,p}(S,T).$$
	\end{Notation}
	\begin{Notation}
			Let $(\Omega,\mathcal{F},\Prob)$ be some probability space. Then we denote by $\Prob^*$ the corresponding outer measure, i.e.
			$$\Prob^*(A):=\inf\left\{\Prob(B):A\subseteq B,B\in\mathcal{F}\right\}, \ A\subseteq\Omega.$$
	\end{Notation}
	\begin{Notation}
		In the sequel, we always equip topological spaces with their corresponding Borel $\sigma$-algebra, and subspaces with the usual subspace topology. Furthermore, if $(E,\mathcal{E})$ is some measurable space, we denote by $B_b(E)$ the space of bounded, measurable functions.
	\end{Notation}
	\begin{Notation}
		If not stated otherwise, $W$ will be a $d$-dimensional Brownian motion on some arbitrary but fixed probability space $(\Omega,\mathcal{F},\Prob)$ and every strong solution shall be defined on this space.
		
		However, weak solutions of equation $\eqref{eq}$ might be defined on different filtrated probability spaces. Therefore, we use the short hand notation $(X^x,\tilde{W}^x,\mathds{Q}^x)$ where $X^x$ is an adapted, continuous stochastic process, $\tilde{W}^x$ is an adapted Brownian motion, both with respect to some filtrated probability space $(\tilde{\Omega},\tilde{\mathcal{F}},\mathds{Q}^x,(\tilde{\mathcal{F}}_t)_{t\geq0})$, and $(X^x,\tilde{W}^x)$ solves equation \eqref{eq} with initial value $x$.
	\end{Notation}
	\begin{Condition}
		\label{singulardriftc}
		Let $p,q>1$ be given with \eqref{pq}. One has for every $T>0$
		$$b\in L^q_p(T).$$
	\end{Condition}
	\begin{Condition}
		\label{sigmac}
		Assume that for all $T>0$ there exists some $C_\sigma=C_\sigma(T)>0$ such that
		\begin{enumerate}
			\item $C_\sigma^{-1}I_{d\times d}\leq\sigma(t,x)\sigma(t,x)^\top\leq C_\sigma I_{d\times d} \ \forall t\in[0,T],x\in\R^d$,
			\item $\norm{\sigma(t,x)-\sigma(t,y)}_{HS}\leq C_\sigma\abs{x-y} \ \forall t\in[0,T],x,y\in\R^d$.
		\end{enumerate}
	\end{Condition}
	\begin{Condition}
		\label{delaydriftc}
		For $t\in[0,r)$ the function $x\mapsto B(t,x)$ is continuous and for all $T>0$ there exists some monotone increasing $g_T:\R_{\geq0}\to\R_{\geq0}$ with
		\begin{enumerate}
			\item $\abs{B(t,x)}\leq g_T(\norm{x}_\infty) \ \forall x\in\Co, t\in[0,T]$,
			\item $\lim\limits_{r\to\infty}g_T(r)/r=0$.
		\end{enumerate}
	\end{Condition}
	\begin{Condition}
		\label{delaydriftlip}
		For all $T>0$ there exists some $C_B=C_B(T)>0$ such that
		$$\abs{B(t,x)-B(t,y)}\leq C_B\norm{x-y}_\infty \ \forall t\in[0,T],x,y\in\Co.$$
	\end{Condition}
	The main results read as follows.
	\begin{Theorem}
		\label{StrongFeller}
		Assume conditions \eqref{singulardriftc}, \eqref{sigmac} and \eqref{delaydriftc}. Then for each initial value $x\in\mathcal{C}$, equation \eqref{eq} has a global weak solution $(X^x,\tilde{W}^x,\mathds{Q}^x)$, which is unique in distribution. Furthermore, one has the strong Feller property for all $t>r$, i.e.
		$$\lim\limits_{y\to x}\E _{\mathds{Q}^y}f(X^y_t)=\E_{\mathds{Q}^x}f(X^x_t) \ \forall f\in B_b(\Co).$$
		Moreover, if condition \eqref{delaydriftlip} is fulfilled, then equation \eqref{eq} has a unique strong solution and it holds
		$$\lim\limits_{y\to x}\E_\Prob\abs{f(X_t^y)-f(X_t^x)}=0 \ \forall f\in B_b(\Co).$$
	\end{Theorem}
	The following theorem is the key element for proving Theorem \ref{StrongFeller}.
	\begin{Theorem}
		\label{goodconv}
		Let $\left(\Omega,\mathcal{F},\Prob\right)$ be some probability space and $\left(E,d\right)$ be a metric space. Furthermore, let $X,X_n:\Omega\to E$, $n\in\N$ be measurable maps. Then the statement
		\begin{enumerate}
			\item
			\begin{enumerate}
				\item$\lim\limits_{n\to\infty}\Prob^*\left(d\left(X,X_n\right)\geq\varepsilon\right)=0 \ \forall\varepsilon>0$,
				\item$\lim\limits_{n\to\infty}\Prob_{X_n}\left(O\right)=\Prob_X\left(O\right),$ for all open $O\subset E$
			\end{enumerate}
		\end{enumerate}
		implies
		\begin{enumerate}
			\item[2.]$\lim\limits_{n\to\infty}\E\abs{f(X)-f(X_n)}=0 \ \forall f\in B_b(\Co).$
		\end{enumerate}
		Additionally, if there exists some null set $N\subset\Omega$ such that $X(\Omega\setminus N)$ is separable, then the converse implication is also true.
	\end{Theorem}
	Moreover, we give a version of Theorem \ref{goodconv} in a topologically more general setup, which reads as follows (the topological terminologies are given in subsection \ref{Preliminaries}).
	\begin{Theorem}
		\label{MainTheorem1}
		Let $\left(\Omega,\mathcal{F},\Prob\right)$ be some probability space and $\left(E,\mathcal{D}\right)$ be a gauge space. Furthermore, let $X,X_n:\Omega\to E$, $n\in\N$ be measurable maps and assume that $\Prob_X$ is outer regular and there exists some null set $N\subset\Omega$ such that $X(\Omega\setminus N)$ is Linde\"of. Then the following statements are equivalent
		\begin{enumerate}
			\item
			\begin{enumerate}
				\item$\lim\limits_{n\to\infty}\Prob^*\left(d\left(X,X_n\right)\geq\varepsilon\right)=0 \ \forall\varepsilon>0,d\in\mathcal{D},$ 
				\item$\lim\limits_{n\to\infty}\Prob_{X_n}\left(O\right)=\Prob_X\left(O\right)$ for all open $O\subset E$.
			\end{enumerate}
			\item $\lim\limits_{n\to\infty}\E\abs{f(X)-f(X_n)}=0 \ \forall f\in B_b(E).$
		\end{enumerate}
	\end{Theorem}
	\begin{Remark}
		To the best of our knowledge, this kind of convergence has not been studied systematically. In this work, it is a key element for proving the strong Feller property together with the approach of Maslowski and Seidler, see \cite{Maslowski2000}. In subsection \ref{Method}, we introduce the underlying strategy, which is applicable to a much more general setup.
		
		Additionally, in subsection \ref{examples} we give some examples to compare the different convergence concepts.
	\end{Remark}
	\begin{Remark}
		The continuity assumption in condition \eqref{delaydriftc} might look artificial. However, the following example illustrates that the strong Feller property is not given in general if one drops this assumption. Consider the SDDE (with $r=1$)
		\begin{align*}
		\dd{X^x}(t)&=\sgn\left(X^x(t-1)\right)\dd{t}+\dd{W}(t),\\
		X_0^x&=x
		\end{align*}
		with the convention
		$$
		\sgn x=
		\begin{cases}
		1 & \text{if }x\geq0,\\
		-1 & \text{if }x<0.
		\end{cases}
		$$
		This equation has for each initial value a unique strong solution, which can be constructed recursively. Now, set $y_n\equiv-1/n$, $n\in\N$, then one has
		\begin{align*}
		X^0(1)&=W(1)+1,\\
		X^{y_n}(1)&=W(1)-1-\frac{1}{n}.
		\end{align*}
		Thus, the strong Feller property is not given.
	\end{Remark}
	\begin{Remark}
		In appendix \ref{stricttop}, we consider the strict topology on the space of bounded, continuous functions as an example of a non-metrizable, locally convex space where all assumptions of Theorem \ref{MainTheorem1} are fulfilled. The strict topology is used for Markov processes with a state space that is not locally compact, see \cite{van2011markov}.
	\end{Remark}
	\section{Convergence of Random Variables in Topological Spaces and Examples}
	\subsection{Preliminaries}
	\label{Preliminaries}
	\begin{Definition}
		A topological space $X$ is called Lindel\"of if every open cover of $X$ has a countable subcover. $X$ is called hereditarily Lindel\"of if every open set of $X$ is Lindel\"of with respect to the subspace topology.
	\end{Definition}
	\begin{Definition}
		Let $E$ be some nonempty set and $\mathcal{D}$ be a nonempty set of pseudometrics on $E$. Then we call $(E,\mathcal{D})$ a gauge space and its topology shall be generated by
		$$\left\{B^d_r(x):x\in E,d\in\mathcal{D},r>0\right\}$$
		where
		$$B_r^d(x):=\left\{y\in E:d(x,y)<r\right\}, \ x\in E,d\in\mathcal{D},r>0.$$
	\end{Definition}
	\begin{Definition}
		A Borel probability measure $\Prob$ on a topological space $(E,\mathcal{T})$ is called outer regular if for all Borel sets $A\in\mathcal{B}(E)$ and $\varepsilon>0$ there exists some open $O\in\mathcal{T}$ with $A\subseteq O$ such that
		$$\Prob(O\setminus A)<\varepsilon.$$
	\end{Definition}
	\subsection{Convergence of Random Variables in Topological Spaces}
	At the beginning we prove the following abstract lemma.
	\begin{Lemma}
		\label{abstractlemma}
		Let $\left(\Omega,\mathcal{F},\Prob\right)$ be some probability space and $\left(E,\mathfrak{E}\right)$ be a measurable space. Furthermore, let $X_n:\Omega\to E$, $n\in\N$ be a sequence of measurable maps, $X:\Omega\to E$ be measurable and $\mathfrak{S}\subseteq\mathfrak{E}$ such that
		$$\forall A\in\mathfrak{E},\varepsilon>0 \ \exists S\in\mathfrak{S}:\Prob_X(A\Delta S)<\varepsilon$$
		where $\Delta$ denotes the symmetric difference of sets. Then the following statements are equivalent
		\begin{enumerate}
			\item 
			\begin{enumerate}
				\item$\lim\limits_{n\to\infty}\Prob\left(X\in S,X_n\notin S\right)=0$ for all $S\in\mathfrak{S}$,
				\item$\lim\limits_{n\to\infty}\Prob_{X_n}\left(A\right)=\Prob_X\left(A\right)$ for all $A\in\mathfrak{E}$.
			\end{enumerate}
			\item $\lim\limits_{n\to\infty}\E\abs{f(X)-f(X_n)}=0 \ \forall f\in B_b(E).$
		\end{enumerate}
	\end{Lemma}
	\begin{proof}
		Implication $2.\Rightarrow1.$ is trivial. Hence, we only show implication $1.\Rightarrow2.$ Since $f$ is bounded, it suffices to prove for all $A\in\mathfrak{E}$
		$$\lim\limits_{n\to\infty}\Prob\left(X\in A, X_n\notin A\right)=0.$$
		Let $\varepsilon>0$. By assumption, there exists some $S\in\mathfrak{S}$ such that
		$$\Prob_X(A\Delta S)<\varepsilon.$$
		Now, choose $n_0\in\N$ large enough such that
		$$\Prob_{X_n}\left(A\Delta S\right)<\varepsilon$$
		for all $n\geq n_0$. Then one has for all $n\geq n_0$
		$$\Prob\left(X\in A,X_n\notin A\right)\leq\Prob\left(X\in S,X_n\notin S\right)+2\varepsilon.$$
		By assumption, one obtains
		$$\lim\limits_{n\to\infty}\Prob\left(X\in A, X_n\notin A\right)\leq2\varepsilon.$$
		Since $\varepsilon>0$ was chosen arbitrarily, the proof is complete.
	\end{proof}
	\begin{Remark}
		Let $(E,\mathcal{T})$ be a topological space and $(\Prob_n)_{n\in\N}$ be a sequence of probability measures that converges pointwise on $\mathcal{T}$ to some outer regular probability measure, i.e.
		$$\lim\limits_{n\to\infty}\Prob_n(O)=\Prob(O) \ \forall O\in\mathcal{T}.$$
		Then one has pointwise convergence on all Borel sets:
		$$\lim\limits_{n\to\infty}\Prob_n(A)=\Prob(A) \ \forall A\in\mathcal{B}(E).$$
		This can be seen as follows. Let $A\in\mathcal{B}(E)$ and $\varepsilon>0$. Since $\Prob$ was assumed to be outer regular, one can find an open set $O\supseteq A$ and a closed set $C\subseteq A$ such that
		$$\Prob(O\setminus C)<\varepsilon.$$
		By assumption, it follows
		$$\Prob(C)=\lim\limits_{n\to\infty}\Prob_n(C)\leq\liminf\limits_{n\to\infty}\Prob_n(A)\leq\limsup\limits_{n\to\infty}\Prob_n(A)\leq\lim\limits_{n\to\infty}\Prob_n(O)=\Prob(O)$$
		and consequently,
		$$\Prob(A)-\varepsilon\leq\liminf\limits_{n\to\infty}\Prob_n(A)\leq\limsup\limits_{n\to\infty}\Prob_n(A)\leq\Prob(A)+\varepsilon \ \forall\varepsilon>0.$$
	\end{Remark}
	Now, one can apply the previous lemma to the topological context, which reads as follows.
	\begin{Corollary}
		\label{convintop}
		Let $\left(\Omega,\mathcal{F},\Prob\right)$ be some probability space and $\left(E,\mathcal{T}\right)$ be a topological space. Furthermore, let $X_n:\Omega\to E$, $n\in\N$ be a sequence of measurable maps and $X:\Omega\to E$ be measurable such that the measure $\Prob_X$ is outer regular. Then the following statements are equivalent
		\begin{enumerate}
			\item for all open $O\in\mathcal{T}$ holds
			\begin{enumerate}
				\item$\lim\limits_{n\to\infty}\Prob\left(X\in O,X_n\notin O\right)=0$,
				\item$\lim\limits_{n\to\infty}\Prob_{X_n}\left(O\right)=\Prob_X\left(O\right)$.
			\end{enumerate}
			\item $\lim\limits_{n\to\infty}\E\abs{f(X)-f(X_n)}=0 \ \forall f\in B_b(E).$
		\end{enumerate}
	\end{Corollary}
	\begin{Remark}
		For a separable metric space $(E,d)$, it holds
		$$\mathcal{B}(E\times E)=\mathcal{B}(E)\otimes\mathcal{B}(E).$$
		If one drops the separability assumption, this may fail. See for example $E=2^\R$ equipped with the discrete topology. Consequently, for two measurable random variables $X,Y:\Omega\to E$, the map
		$$\Omega\ni\omega\mapsto d(X(\omega),Y(\omega))$$
		could be not measurable. To overcome this problem, one can use the outer measure
		$$\Prob^*\left(A\right):=\inf\left\{\Prob(B):A\subset B,\ B\text{ measurable}\right\}$$
		to evaluate
		$$\Prob^*\left(d(X,Y)\geq\varepsilon\right), \ \varepsilon>0.$$
		This provides a natural definition of convergence in probability for non-separable metric spaces. For a discussion in detail, see \cite{Empirical}.
	\end{Remark}
	\begin{Lemma}
		\label{convinprob}
		Let $\left(\Omega,\mathcal{F},\Prob\right)$ be some probability space and $\left(E,\mathcal{D}\right)$ be a gauge space. Furthermore, let $X,X_n:\Omega\to E$, $n\in\N$ be measurable maps and assume that $\Prob_X$ is outer regular and there exists some null set $N\subset\Omega$ such that $X(\Omega\setminus N)$ is Lindel\"of. Then the following statements are equivalent
		\begin{enumerate}
			\item $\lim\limits_{n\to\infty}\Prob^*\left(d\left(X,X_n\right)\geq\varepsilon\right)=0 \ \forall\varepsilon>0,d\in \mathcal{D}$,
			\item $\lim\limits_{n\to\infty}\Prob\left(X\in O,X_n\notin O\right)=0$ for all open $O\subset E$. 
		\end{enumerate}
	\end{Lemma}
	\begin{proof}
		$1.\Rightarrow2.:$ Without loss of generality, one can assume
		$$\max(d_1,d_2)\in\mathcal{D} \ \forall d_1,d_2\in\mathcal{D}.$$
		Let $O\subseteq E$ be open and $\delta>0$. The probability measure $\Prob_X$ was assumed to be outer regular. Thus, there exists an open set $V$ with $E\setminus O\subseteq V$ and
		$$\Prob_X\left(V\cap O\right)<\delta.$$
		Additionally, for every $x\in O$ there exist some pseudometric $d_x\in\mathcal{D}$ and $r_x>0$ such that
		$$O=\bigcup_{x\in O}B^{d_x}_{r_x}(x).$$
		Then it follows
		$$E=\bigcup_{x\in O}B^{d_x}_{r_x}(x)\cup V,$$
		and since $X(\Omega\setminus N)$ is Lindel\"of, one can find $d_i\in\mathcal{D}$, $r_i>0$, $x_i\in E$, $i\in\N$ such that
		$$\Prob_X\left(O\setminus\bigcup_{i=1}^m B_{r_i}^{d_i}(x_i)\right)<\delta.$$
		Then it holds
		\begin{align*}
			\Prob\left(X\in O,X_n\notin O\right)&\leq\sum_{i=1}^{m}\Prob\left(X\in B_{r_i}^{d_i}(x_i),X_n\notin O\right)+\delta\\
			&\leq\sum_{i=1}^{m}\Prob\left(X\in B_{r_i}^{d_i}(x_i),X_n\notin B_{r_i}^{d_i}(x_i)\right)+\delta
		\end{align*}
		Hence, it suffices to show
		$$\lim\limits_{n\to\infty}\Prob\left(X\in B_r^d(x),X_n\notin B_r^d(x)\right)=0 \ \forall d\in\mathcal{D},x\in E, r>0.$$
		Let $d\in\mathcal{D}$, $x\in E$ and $r>0$. By assumption, one has
		$$\lim\limits_{n\to\infty}\Prob^*\left(d\left(X,X_n\right)\geq\varepsilon\right)=0 \ \forall\varepsilon>0.$$
		Now, we show that every subsequence of $(X_n)_{n\in\N}$ has a subsequence that converges pointwise to $X$ with respect to $d$ almost surely. This can be seen as follows. Let $(X_{n_k})_{k\in\N}$ be a subsequence of $(X_n)_{n\in\N}$. Now, for each $l\in\N$ choose $k_l\in\N$ large enough such that $k_l\geq k_{l-1}$ and
		$$\Prob^*\left(d\left(X,X_{k_l}\right)\geq\frac{1}{l}\right)\leq\frac{1}{l^2}.$$
		This is possible by the equality above. Then one has by the properties of an outer measure
		\begin{align*}
			&\Prob^*\bigg(X_{k_l}\nrightarrow X\bigg)\\
			\leq&\Prob^*\left(\bigcup_{m\in\N}\bigcap_{l_0\in\N}\bigcup_{l\geq l_0}\left\{d\left(X, X_{n_{k_l}}\right)\geq\frac{1}{m}\right\}\right)\\
			\leq&\sum_{m\in\N}\Prob^*\left(\bigcap_{l_0\in\N}\bigcup_{l\geq l_0}\left\{d\left(X, X_{n_{k_l}}\right)\geq\frac{1}{m}\right\}\right).
		\end{align*}
		Furthermore, one has
		\begin{align*}
			&\Prob^*\left(\bigcap_{l_0\in\N}\bigcup_{l\geq l_0}\left\{d\left(X, X_{n_{k_l}}\right)\geq\frac{1}{m}\right\}\right)\\
			=&\Prob^*\left(\bigcap_{l_0\in\N}\bigcup_{l\geq l_0}\left\{d\left(X, X_{n_{k_l}}\right)\geq\frac{1}{l}\right\}\right)\\
			\leq&\lim\limits_{l_0\to\infty}\sum_{l=l_0}^{\infty}\Prob^*\left(d\left(X,X_{k_l}\right)\geq\frac{1}{l}\right)\\
			\leq&\lim\limits_{l_0\to\infty}\sum_{l=l_0}^{\infty}\frac{1}{l^2}\\
			=&0.
		\end{align*}
		Consequently, every subsequence of $(X_n)_{n\in\N}$ has a subsequence that converges pointwise to $X$ with respect to $d$ almost surely. 
		Finally, it follows
		$$\lim\limits_{n\to\infty}\Prob\left(X\in B_r^d(x),X_n\notin B_r^d(x)\right)=0.$$
		$2.\Rightarrow1.:$ Let $\varepsilon$, $\delta>0$ and $d\in\mathcal{D}$. Since $X(\Omega\setminus N)$ was assumed to be Lindel\"of, one can choose $m\in\N$ large enough such that
		$$\Prob_X\left(\bigcup_{i=1}^m B_\varepsilon^d(x_i)\right)>1-\delta$$
		with suitable $x_1,\dots,x_m\in E$. Then one has
		\begin{align*}
			&\Prob^*\left(d\left(X,X_n\right)\geq\varepsilon\right)\\
			\leq&\Prob^*\left(X\in\bigcup_{i=1}^mB_\varepsilon^d(x),d(X,X_n)\geq\varepsilon\right)+\Prob^*\left(X\notin\bigcup_{i=1}^mB_\varepsilon^d(x_i),d(X,X_n)\geq\varepsilon\right)\\
			\leq&\sum_{i=1}^{m}\Prob^*\left(X\in B_\varepsilon^d(x_i),d\left(X,X_n\right)\geq\varepsilon\right)+\Prob\left(X\notin\bigcup_{i=1}^mB_\varepsilon^d(x_i)\right)\\
			\leq&\sum_{i=1}^{m}\Prob\left(X\in B_\varepsilon^d(x_i),X_n\notin B_\varepsilon^d(x_i)\right)+\delta
		\end{align*}
		By assumption, it follows
		$$\lim\limits_{n\to\infty}\Prob^*\left(d\left(X,X_n\right)\geq\varepsilon\right)\leq\delta$$
		for all $\delta>0$, which completes the proof.
	\end{proof}
	\begin{Definition}
		In a topological space a $G_\delta$-set is an intersection of countably many open sets. A topological space is called a $G_\delta$-space if every closed set is a $G_\delta$-set.
	\end{Definition}
	\begin{Lemma}
		\label{outerreg}
		Let $(E,\mathcal{T})$ be a $G_\delta$-space. Then every Borel probability measure $\Prob$ on $E$ is outer regular.
	\end{Lemma}
	\begin{proof}
		The proof is standard and can be found for polish spaces in \cite[p. 224-225]{dudley}. Consider the set $\mathfrak{A}$
		$$\mathfrak{A}:=\left\{A\in\mathcal{B}(E):A\text{ and }E\setminus A\text{ outer regular}\right\}.$$
		Clearly, $\mathfrak{A}$ is a $\sigma$-algebra. Therefore, it is sufficient to show that every closed $C\subseteq E$ is contained in $\mathfrak{A}$. By assumption, each closed set is a countable intersection of open sets, which completes the proof.
	\end{proof}
	\begin{Remark}
		Examples for $G_\delta$-spaces are metric spaces and hereditarily Lindel\"of gauge spaces.
	\end{Remark}
	\begin{proof}[Proof of Theorem \ref{goodconv}]
		By Corollary \ref{convintop} and Lemma \ref{convinprob}, it is sufficient to show
		$$\lim\limits_{n\to\infty}\Prob\left(X\in O,X_n\notin O\right)=0 \ \forall\text{ open }O\subset E.$$
		Furthermore,
		$$\lim\limits_{n\to\infty}\Prob^*\left(d(X,X_n)\leq\varepsilon\right)=0 \ \forall\varepsilon>0$$
		implies that every subsequence of $(X_n)_{n\in\N}$ has a subsequence that converges almost surely. Thus, it follows
		$$\lim\limits_{n\to\infty}\1_{O}(X)\1_{E\setminus O}(X_n)=0\text{ in probability},$$
		which proofs the implication. Under the additional assumption, one can apply Lemma \ref{convinprob} again to conclude the reverse direction.
	\end{proof}
	\begin{proof}[Proof of Theorem \ref{MainTheorem1}]
		This is a consequence of Corollary \ref{convintop}, Lemma \ref{convinprob} and Lemma \ref{outerreg}.
	\end{proof}
	\subsection{Examples}
	\label{examples}
	In this subsection we want to give some examples for the reader's convenience. At first, two rather trivial examples are given to discuss the equivalence from Theorem \ref{goodconv}. Then we consider two one-dimensional SD(D)Es to investigate the difference between the strong Feller property and its ``improved'' version.
	\begin{enumerate} 
		\item In this simple example we show that pointwise convergence of random variables does not imply the convergence type discussed in this section. However, if one adds Gaussian terms, the convergence follows. Consider a sequence $(x_n)_{n\in\N}$ in $\R^d$ that converges to some $x_0\in\R^d$ with $x_i\neq x_0$ for all $i\in\N$. Then the deterministic random variables
		$$X_n\equiv x_n$$
		converge pointwise to $x_0$ but their laws $\delta_{x_n}$ do not converge pointwise to $\delta_{x_0}$. In particular, it holds
		$$\E\abs{\1_{\{x_0\}}(X)-\1_{\{x_0\}}(X_n)}=1.$$
		On the other hand, let $N$ be a standard Gaussian random variable and consider instead the sequence
		$$Y_n:=x_n+N.$$
		Then one has
		$$\lim\limits_{n\to\infty}\E\abs{f(Y_n)-f(Y)}=0 \ \forall f\in B_b(\R).$$
		\item This example is rather trivial. It shows that pointwise convergence of measures does not imply the convergence type discussed in this section. Let $X\sim\mathcal{N}(0,1)$ be a standard Gaussian random variable and define
		$$X_n:=-X\sim\mathcal{N}(0,1).$$
		It holds
		$$\Prob(X\in A)=\Prob(X_n\in A) \ \forall n\in\N, A\in\mathcal{B}(\R).$$
		Obviously, $X_n$ does not converge to $X$ in probability and it holds
		$$\E\abs{\1_{\R_{\geq0}}(X)-\1_{\R_{\geq0}}(X_n)}=1 \ \forall n\in\N.$$
		\item In this example we show that pathwise uniqueness and the strong Feller property already implies almost sure convergence for solutions of one-dimensional non-delay equations. Consequently, in this setting, there is no difference between the strong Feller property and it's ``improved'' version. Now, assume we have a one-dimensional SDE that has a unique strong solution for each real initial value
		\begin{align*}
		\dd{X^x}(t)&=b(t,X^x(t))\dd{t}+\sigma(t,X^x(t))\dd{W}(t),\\
		X^x(0)&=x\in\R
		\end{align*}
		where $W$ is a $d$-dimensional Brownian motion on some probability space, and $b:\R_{\geq0}\times \R\to\R$ and $\sigma:\R_{\geq0}\times\R\to\R^{1,d}$ are measurable. Assume furthermore that $X$ has the Feller property, i.e.
		$$\lim\limits_{y\to x}\E f(X^y(t))=\E f(X^x(t)) \ \forall f\in C_b(\R).$$
		Then for every sequence $(x_n)_{n\in\N}\subset\R$ with $x_n\to x$ and $t>0$, one has
		$$X^{x_n}(t)\to X^x(t)\text{ a.s.}$$
		In particular, by Theorem \ref{goodconv}, the strong Feller property is equivalent to
		$$\lim\limits_{n\to\infty}\E\abs{f(X^{x_n}(t))-f(X^x(t))}=0 \ \forall f\in B_b(\R).$$
		This can be seen as follows: by uniqueness, one has monotonicity for the solutions, i.e. for all $x\leq y$ holds
		$$X^x(t)\leq X^y(t) \ \forall t\geq0\text{ a.s.}$$
		On the other hand, for each sequence $(x_n)_{n\in\N}$ with $x_n\downarrow x$, the following limit exists
		$$\tilde{X}(t):=\lim\limits_{x_n\downarrow x}X^{x_n}(t) \ \forall t\geq0\text{ a.s.}$$
		since all $X^{x_n}$ are bounded from below by $X^x$. Thus, one has for all $t\geq0$, $f\in C_b(\R)$
		$$\E f(X^x(t))=\lim\limits_{x_n\downarrow x}\E f(X^{x_n}(t))=\E f(\tilde{X}(t)).$$
		Additionally, it holds
		$$\tilde{X}(t)\geq X^x(t) \ \forall t\geq0\text{ a.s.}$$
		Thus, $\tilde{X}$ is a modification of $X^x$.
		\item
		In this example we show that the implication from the previous example is not given if the dispersion coefficient depends on the past. Especially, condition \eqref{sigmac} is not fulfilled. For the sake of overview, we embed real constants in $\Co$ naturally. 
		Now, let us consider the one-dimensional SDDE (with $r=1$)
		\begin{align*}
		\dd{X^x}(t)&=\sgn(X^x(t-1))\dd{W}(t)\\
		X_0&=x\in\Co
		\end{align*}
		where we use the convention
		$$\sgn x=
		\begin{cases}
		1 &\text{ if }x\geq0,\\
		-1 &\text{ if }x<0.
		\end{cases}
		$$
		This SDDE can be solved uniquely by constructing the solution recursively. By Levy's characterization, each solution $X^x$ is distributed on $\R_{\geq0}$ like a shifted Brownian motion, in particular
		$$X_t^x\sim W_t+x(0) \ \forall t>1.$$
		It is not difficult to show that one has for $t>1$
		$$\lim\limits_{y\to x}\E f(X_t^y)=\lim\limits_{y\to x}\E f(W_t+y(0))=\E f(W_t+x(0))=\E f(X_t^x) \ \forall f\in B_b(\R),$$
		see for example \cite{es-sarhir2009}. So, $X$ has the strong Feller property with respect to the state space $\Co$.
		On the other hand, one has for all $y\geq0,x<0$
		$$\norm{X^y_2-X^x_2}_\infty\geq\abs{X^y(1)-X^x(1)}=\abs{2W(1)+y-x}\text{ a.s.}$$
		Therefore, convergence in probability is not given. In particular, the strong Feller property does not coincide with it's ``improved'' version. 
	\end{enumerate}
	\section{Application to Stochastic Delay Differential Equations}
	\subsection{Introduction to the Method}
	\label{Method}
	In this subsection, we want to illustrate the strategy to prove Theorem \ref{StrongFeller}, which is based on an approach of Maslowski and Seidler, see \cite{Maslowski2000}, and the convergence discussed in section 2. In order, we make use of a toy example with state space $\R^d$. Consider the equation
	\begin{align*}
		\dd{X^x}(t)&=b(t,X^x(t))\dd{t}+\sigma\dd{W}(t),\\
		X^x(0)&=x\in\R^d
	\end{align*}
	where $W$ is some d-dimensional Brownian motion, $\sigma\in\R^{d\times d}$ is invertible, $b:\R_{\geq0}\times\R^d\to\R^d$ is measurable and for every $T>0$ there exists some $C_T\in\R$ with
	\begin{align*}
		\left\langle b(t,u)-b(t,v),u-v\right\rangle&\leq C_T\abs{u-v}^2 \ \forall u,v\in\R^d,t\in[0,T],\\
		\abs{b(t,u)}&\leq C_T \ \forall u\in\R^d,t\in[0,T].
	\end{align*}
	Additionally, we consider its drift-free equation
	\begin{align*}
		\dd{M^x}(t)&=\sigma\dd{W}(t),\\
		M^x(0)&=x\in\R^d,
	\end{align*}
	which is trivial in that case for the sake of simplicity. Observe that both equations have a unique strong solution and the drift-free one depends continuously on the initial value in the sense that
	$$\lim\limits_{n\to\infty}M^{x_n}(t)=M^x(t)\text{ in probability}$$
	for each sequence $(x_n)_{n\in\N}\subset\R^d$ converging to $x$. Now, one can observe that $M$ has the strong Feller property, i.e.
	$$\lim\limits_{y\to x}\E f(M^y(t))=\E f(M^x(t)) \ \forall f\in B_b(\R^d).$$
	Also, $\Prob_{X^x}$ has a Girsanov density with respect to $\Prob_{M^x}$, i.e.
	$$\E f(X^x(t))=\E \left[D^x(t)f(M^x(t))\right] \ \forall f\in B_b(\R^d)$$
	with
	$$D^x(t)=\exp\left(\int_{0}^{t}\sigma^{-1}b(s,M^x(s))^\top\dd{W}(s)-\frac{1}{2}\int_{0}^{t}\abs{\sigma^{-1}b(s,M^x(s))}^2\dd{s}\right).$$
	At first, we show that $X$ has the strong Feller property, too. Let $f\in B_b\left(\R^d\right)$, then one has
	\begin{align*}
		&\E f(X^x(t))-\E f(X^y(t))\\
		=&\E[D^x(t)f(M^x(t))]-\E[D^y(t)f(M^y(t))]\\
		\leq&\E[D^x(t)(f(M^x(t))-f(M^y(t)))]+\norm{f}_\infty\E\abs{D^x(t)-D^y(t)}.
	\end{align*}
	By Theorem \ref{goodconv}, one has
	$$\lim\limits_{y\to x}\E\abs{f(M^x(t))-f(M^y(t))}=0$$
	and in particular,
	$$\lim\limits_{n\to\infty}D^x(t)f(M^{x_n}(t))=D^x(t)f(M^x(t))\text{ in probability}$$
	for each sequence $(x_n)_{n\in\N}\subset\R^d$ converging to $x$. By the dominated convergence theorem, it follows
	$$\lim\limits_{y\to x}\E[D^x(t)(f(M^y(t))-f(M^x(t)))]=0.$$
	Consequently, it remains to show
	$$\lim\limits_{y\to x}\E\abs{D^x(t)-D^y(t)}=0.$$
	Assume for a moment that one has
	$$\lim\limits_{n\to\infty}D^{x_n}(t)=D^x(t)\text{ in probability}$$
	for each sequence $(x_n)_{n\in\N}\subset\Co$ converging to $x$. Since one has $\E_\Prob D^z(t)=1$ and $D^z(t)\geq0$ for all $z\in\Co$, one could apply Fatou's lemma to conclude
	$$2-\lim\limits_{n\to\infty}\E_\Prob\abs{D^{x_n}(t)-D^x(t)}=\lim\limits_{n\to\infty}\E_\Prob\left(D^x(t)+D^{x_n}(t)-\abs{D^{x_n}(t)-D^x(t)}\right)\geq2$$
	and the desired $L^1$-convergence follows. Thus, it suffices to show
	$$\lim\limits_{n\to\infty}D^{x_n}(t)=D^x(t)\text{ in probability}$$
	for each sequence $(x_n)_{n\in\N}\subset\Co$ converging to $x$. Therefore, it is sufficient to show
	$$\lim\limits_{y\to x}\E\int_{0}^{t}\abs{b(s,M^y(s))-b(s,M^x(s))}^2\dd{s}=0,$$
	by the martingale isometry. However, this is a direct consequence of Theorem \ref{goodconv}. Finally, one ends up with
	$$\lim\limits_{y\to x}\E f(X^y(t))=\E f(X^x(t)).$$
	By It\=o's formula and Gronwall's lemma, one can easily show
	$$\lim\limits_{n\to\infty}X^{x_n}(t)=X^x(t)\text{ in probability}$$
	for each sequence $(x_n)_{n\in\N}\subset\R^d$ converging to $x$. Thus, it even follows
	$$\lim\limits_{y\to x}\E\abs{f(X^y)(t)-f(X^x(t))}=0$$
	by Theorem \ref{goodconv}.
	
	The well-posedness and the strong Feller property of the drift-free equation, the existence of densities of $X^x$ with respect to $M^x$, $x\in\R^d$, and their convergence in probability are exactly the requirements Maslowski and Seidler needed for one of their approaches to show the strong Feller property, cf. Theorem 2.1 in \cite{Maslowski2000}. Theorem \ref{goodconv} systematically extends their approach by showing the convergence of the densities even for discontinuous drift coefficients.
	
	 Now, we can summarize the strategy for showing the strong Feller property in a few steps without specifying the details: show that
	\begin{enumerate}
		\item the drift-free version of the original equation has a unique strong solution $M^x$ for each initial value $x$.
		\item $M$ has the strong Feller property and for every sequence $(x_n)_{n\in\N}$ with $x_n\to x$, one has
		$$\lim\limits_{n\to\infty}M^{x_n}_t=M^x_t\text{ in probability}$$
		\item The equation with drift has for each initial value $x$ a weak solution $X^x$ that is unique in distribution.
		\item For every initial value $x$, $\Prob_{X^x}$ has a density with respect to $\Prob_{M^x}$ such that for any sequence $(x_n)_{n\in\N}$ with $x_n\to x$, one has
		$$\lim\limits_{x_n\to x}\frac{\dd{\Prob_{X^{x_n}}}}{\dd{\Prob_{M^{x_n}}}}(M^{x_n})=\frac{\dd{\Prob_{X^x}}}{\dd{\Prob_{M^x}}}(M^x)\text{ in probability}.$$
		As illustrated by the previous example, Theorem \ref{goodconv} is the key element for verifying this step since one has to deal with discontinuous coefficients.
	\end{enumerate}
	If, in addition, the equation with drift has for each initial value $x$ a unique strong solution $X^x$ such that for every sequence $(x_n)_{n\in\N}$ with $x_n\to x$, one has
	$$\lim\limits_{n\to\infty}X^{x_n}_t=X^x_t\text{ in probability,}$$
	then one can apply Theorem \ref{goodconv} again to deduce the ``improved'' version of the Feller property.
	\subsection{A-priori Estimates, Uniqueness and Existence}
	In the sequel, denote by $M^x$, $x\in\mathcal{C}$ the global, unique strong solution of
	\begin{align*}
		\dd{M^x}(t)&=\sigma\left(t,M^x(t)\right)\dd{W}(t),\\
		M_0^x&=x.
	\end{align*}
	\begin{Remark}
		Condition \eqref{delaydriftc} implies the following important property of the delay drift $B$: for all $\alpha>0$ and $T>0$ there exists a $K_{\alpha,T}>0$ such that
		$$\abs{B(t,x)}\leq \alpha\norm{x}+K_{\alpha,T} \ \forall t\in[0,T].$$
		In the sequel, this property will be exploited to show that the Novikov-condition is fulfilled for various Girsanov densities. 
		
		Moreover, condition \eqref{sigmac} implies the following inequalities
		$$\norm{\sigma}_{op},\norm{\sigma^{-1}}_{op}\leq\sqrt{C_\sigma}.$$
	\end{Remark}
	\begin{Lemma}
		\label{Krylov}
		Assume condition \eqref{sigmac}. Let $T>0$ and $p',q'>1$ be given with
		$$\frac{d}{p'}+\frac{2}{q'}<2.$$
		Then one has for all $0\leq S<T$ and $f\in L_{p'}^{q'}(S,T)$ the estimate
		$$\E\left(\int_{S}^{T}f(t,M^x(t))\dd{t}\bigg|\mathcal{F_S}\right)\leq C\norm{f}_{L_{p'}^{q'}(S,T)}$$
		for some constant $C=C(d,p',q',T,C_\sigma)$. In particular, the constant $C$ is independent of the initial value $x\in\Co$.
	\end{Lemma}
	\begin{proof}
		This follows directly from Theorem 2.1 in \cite{Zhang2011}.
	\end{proof}
	\begin{Lemma}
		\label{expfM}
		Assume condition \eqref{sigmac}. Then for any $R,T>0$ and $p',q'>1$ with
		$$\frac{d}{p'}+\frac{2}{q'}<2$$
		there exists a constant $C_R=C_R(d,p',q',T,C_\sigma)$ such that
		$$\E\exp\left(\int_{0}^{T}f(t,M^x(t))\dd{t}\right)\leq C_R$$
		for all $f\in L_{p'}^{q'}(T)$ with $\norm{f}_{L_{p'}^{q'}(T)}\leq R$.
	\end{Lemma}
	\begin{proof}
		See Lemma 2.2 in \cite{Stefan1}.
	\end{proof}
	\begin{Lemma}
		\label{expsupM}
		Assume condition $\eqref{sigmac}$. Then for any $T>0$ and $0\leq\alpha<(2dC_\sigma T)^{-1}$, it holds
		$$\E\exp\left(\alpha\sup_{0\leq t\leq T}\abs{M^x(t)}^2\right)\leq\frac{4}{\sqrt{1-2\alpha dC_\sigma T}}\exp\left(\frac{\alpha}{1-2\alpha dC_\sigma T}\abs{x(0)}^2\right).$$
	\end{Lemma}
	\begin{proof}
		See Lemma 2.4 in \cite{Stefan1}.
	\end{proof}
	\begin{Theorem}
		\label{ExUni}
		Assume conditions \eqref{singulardriftc}, \eqref{sigmac} and \eqref{delaydriftc}. Then for every initial values $x\in\Co$, equation \eqref{eq} has a global weak solution. Moreover, for each weak solution $(X^x,\tilde{W}^x,\mathds{Q}^x)$ of equation \eqref{eq} on some time interval $[-r,T]$, $T>0$, one has
		\begin{align*}
		\mathds{Q}_{X^x}^x(A)&=\E_\Prob\bigg[\1_A(M^x)\exp\bigg(\int_{0}^{T}a^x(t)^\top \dd{W}(t)-\frac{1}{2}\int_{0}^{T}\abs{a^x(t)}^2\dd{t}\bigg)\bigg],\\
		a^x(t)&:=\sigma(t,M^x(t))^{-1}\left[B(t,M^x_t)+b(t,M^x(t))\right], \ t\in[0,T]
		\end{align*}
		for all measurable $A\subset C([-r,T],\R^d)$. In addition, if condition \eqref{delaydriftlip} holds, equation \eqref{eq} admits a unique strong solution.
	\end{Theorem}
	\begin{proof}
		At first, we show the existence of a weak solution. The additional statement about the pathwise uniqueness has been shown in \cite{Stefan1}, Theorem 1.5, and the existence of a strong solution then follows from the Theorem of Yamada and Watanabe \cite{yamada1971}.
		
		The strong solution $M^x$ is by definition $\left(\mathcal{F}_t\right)_{t\geq0}$-adapted where $\left(\mathcal{F}_t\right)_{t\geq0}$ is the augmented filtration generated by $W$. Next, we construct a probability measure on
		$$\mathcal{F}_\infty:=\sigma\left(\mathcal{F}_t:t\geq0\right)$$
		such that $M^x$ is a global weak solution for equation \eqref{eq}. By Lemma \ref{expfM}, Lemma \ref{expsupM}, condition \eqref{sigmac} and condition \eqref{delaydriftc}, it holds
		\begin{align*}
			\E_\Prob\exp\left(\alpha\int_{0}^{T}\abs{\sigma(t,M^x(t))^{-1}B(t,M^x_t)}^2\dd{t}\right)&<\infty,\\
			\E_\Prob\exp\left(\alpha\int_{0}^{T}\abs{\sigma(t,M^x(t))^{-1}b(t,M^x(t))}^2\dd{t}\right)&<\infty
		\end{align*}
		for all $\alpha\in\R$ and $T>0$. Therefore, Novikov's condition is fulfilled and Girsanov's theorem is applicable, which gives that
		$$\bar{W}(t):=W(t)-\int_{0}^{t}\sigma(s,M^x(s))^{-1}a^x(s)\dd{s}, \ t\geq0$$
		is a Brownian motion on $[0,T]$ under the probability measure
		\begin{align*}
		\dd{\bar{\Prob}_T}:=&\exp\bigg(\int_{0}^{T}\left(\sigma(t,M^x(t))^{-1}a^x(t)\right)^\top\dd{W}(t)\\
		&-\frac{1}{2}\int_{0}^{T}\abs{\sigma(t,M^x(t))a^x(t)}^2\dd{t}\bigg)\dd{\Prob}
		\end{align*}
		and $(M^x,\bar{W},\bar{\Prob}_T)$ is a weak solution of \eqref{eq} on $[-r,T]$ for each $T>0$. Additionally, one has for $0<T_1<T_2$
		$$\bar{\Prob}_{T_1}(A)=\bar{\Prob}_{T_2}(A) \ \forall A\in\mathcal{F}_{T_1},$$
		so the probability measure on $\mathcal{F}_\infty$ uniquely defined by
		$$\bar{\Prob}(A):=\Prob_T(A) \ \forall T>0, A\in\mathcal{F}_T$$
		is indeed well-defined and $(M^x,\bar{W},\bar{\Prob})$ is a global weak solution.
		
		Now, let $(X^x,\tilde{W}^x,\mathds{Q}^x)$ be a weak solution on some time interval $[0,T]$, $T>0$. The following approach is inspired by the techniques used in \cite{liptser2001statistics}. Define
		$$\tau^n(\omega):=\inf\left\{s\geq0:\abs{\omega(s)}\geq n\right\}\wedge T, \ \omega\in C([-r,T],\R^d),\ n\in\N.$$
		Then the stopped process $X^{x,n}(t):=X^x(t\wedge\tau^n(X^x))$, $t\in[-r,T]$ fulfills the equation
		\begin{align*}
			\dd{X^{x,n}}(t)=\1_{\tau^n(X^{x,n})\leq t}\left[B(t,X_t^{x,n})+b(t,X^{x,n}(t))\right]\dd{t}+\1_{\tau^n(X^{x,n})\leq t}\sigma(t,X^{x,n}(t))\dd{\tilde{W}^x}
		\end{align*}
		By condition \eqref{delaydriftc} and Girsanov's theorem,
		$$\tilde{W}^{x,n}(t):=\int_{0}^{t\wedge\tau^n(X^{x,n})}\sigma(s,X^{x,n}(s))^{-1}B(s,X^{x,n}_s)\dd{s}+\tilde{W}^x(t), \ t\geq0$$
		is a Brownian motion with respect to the probability measure
		\begin{align*}
			\dd{\mathds{Q}^{x,n}}:=&\exp\bigg(-\int_{0}^{\tau^n(X^{x,n})}\left(\sigma(t,X^{x,n}(t))^{-1}B(t,X^{x,n}_t)\right)^\top\dd{\tilde{W}^x}(t)\\
			&-\frac{1}{2}\int_{0}^{\tau^n(X^{x,n})}\abs{\sigma(t,X^{x,n}(t))B(t,X^{x,n}_t)}^2\dd{t}\bigg)\dd{\mathds{Q}^x}.
		\end{align*}
		The process $X^{x,n}$ solves the equation
		\begin{align*}
			\dd{X^{x,n}}(t)&=b(t,X^{x,n}(t))\dd{t}+\sigma(t,X^{x,n}(t))\dd{\tilde{W}^{x,n}}(t), \ t\in[0,\tau^n(X^{x,n})],\\
			X^{x,n}_0&=x.
		\end{align*}
		Such a solution is unique by Theorem 1.3 in \cite{Zhang2011}, i.e.
		$$X^{x,n}(t)=Y^{x,n}(t), \ t\in[-r,\tau^n(X^{x,n})]$$
		where $Y^{x,n}$ is the unique strong solution of
		\begin{align*}
			\dd{Y^{x,n}}(t)&=b(t,Y^{x,n}(t))\dd{t}+\sigma(t,Y^{x,n}(t))\dd{\tilde{W}^{x,n}}(t),\\
			Y^{x,n}_0&=x
		\end{align*}
		and it holds
		$$\tau^n(X^{x,n})=\tau^n(Y^{x,n})\text{ a.s.}$$
		Consequently, the process $X^{x,n}$ is distributed, with respect to $\mathds{Q}^{x,n}$, as the stopped process $t\mapsto Y^{x,n}(t\wedge\tau^n(Y^{x,n}))$, $t\geq-r$. It follows
		\begin{align*}
			&\mathds{Q}^x(X^x\in A)\\
			=&\lim\limits_{n\to\infty}\mathds{Q}^x(\tau^n(X^x)=T,X^x\in A)\\
			=&\lim\limits_{n\to\infty}\mathds{Q}^x(\tau^n(X^{x,n})=T,X^{x,n}\in A)\\
			=&\lim\limits_{n\to\infty}\E_{\mathds{Q}^{x,n}}\bigg[\1_{\tau^n(X^{x,n})=T}\1_A(X^{x,n})\exp\bigg(\int_{0}^{T}\left(\sigma(t,X^{x,n}(t))^{-1}B(t,X^{x,n}_t)\right)^\top\dd{\tilde{W}^{x,n}}(t)\\
			&-\frac{1}{2}\int_{0}^{T}\abs{\sigma(t,X^{x,n}(t))^{-1}B(t,X_t^{x,n})}^2\dd{t}\bigg)\bigg]\\
			=&\lim\limits_{n\to\infty}\E_{\mathds{Q}^{x,n}}\bigg[\1_{\tau^n(Y^{x,n})=T}\1_A(Y^{x,n})\exp\bigg(\int_{0}^{T}\left(\sigma(t,Y^{x,n}(t))^{-1}B(t,Y_t^{x,n})\right)^\top\dd{\tilde{W}^{x,n}}(t)\\
			&-\frac{1}{2}\int_{0}^{T}\abs{\sigma(t,Y^{x,n}(t))^{-1}B(t,Y_t^{x,n})}^2\dd{t}\bigg)\bigg]
		\end{align*}
		for all measurable $A\subset C([-r,T],\R^d)$. On the other hand, by Lemma \ref{expfM}, the process
		$$\hat{W}(t):=W(t)-\int_{0}^{t}\sigma(s,M^x(s))^{-1}b(s,M^x(s))\dd{s}, \ t\geq0$$
		is a Brownian motion under the probability measure
		\begin{align*}
			\dd{\hat{\Prob}}:=&\exp\bigg(\int_{0}^{T}\left(\sigma(t,M^x(t))^{-1}b(t,M^x_t)\right)^\top\dd{W}(t)\\
			&-\frac{1}{2}\int_{0}^{T}\abs{\sigma(t,M^x(t))b(t,M^x_t)}^2\dd{t}\bigg)\dd{\Prob}.
		\end{align*}
		and $M^x$ solves the equation
		\begin{align*}
			\dd{M^x}(t)&=b(t,M^x(t))\dd{t}+\sigma(t,M^x(t))\hat{W}(t)\\
			M^x_0&=x.
		\end{align*}
		Again, by uniqueness in distribution, one has $\hat{\Prob}_{M^x}=\mathds{Q}_{Y^{x,n}}^{x,n}$ for all $n\in\N$. With this in hand, one obtains
		\begin{align*}
			&\E_\Prob\bigg[\1_A(M^x)\exp\bigg(\int_{0}^{T}a^x(t)^\top \dd{W}(t)-\frac{1}{2}\int_{0}^{T}\abs{a^x(t)}^2\dd{t}\bigg)\bigg]\\
			=&\lim\limits_{n\to\infty}\E_\Prob\bigg[\1_{\tau^n(M^x)=T}\1_A(M^x)\exp\bigg(\int_{0}^{T}a^x(t)^\top \dd{W}(t)-\frac{1}{2}\int_{0}^{T}\abs{a^x(t)}^2\dd{t}\bigg)\bigg]\\
			=&\lim\limits_{n\to\infty}\E_{\hat{\Prob}}\bigg[\1_{\tau^n(M^x)=T}\1_A(M^x)\\
			&\times\exp\bigg(\int_{0}^{T}B(t,M^x_t)^\top\left(\sigma(t,M^x(t))\sigma(t,M^x(t))^\top\right)^{-1}\dd{M^x}(t)\\
			&-\frac{1}{2}\int_{0}^{T}(2b(t,M^x(t))+B(t,M^x_t))^\top\left(\sigma(t,M^x(t))\sigma(t,M^x(t))^\top\right)^{-1}B(t,M^x_t)\dd{t}\bigg)\bigg]\\
			=&\lim\limits_{n\to\infty}\E_{\mathds{Q}^{x,n}}\bigg[\1_{\tau^n(Y^{x,n})=T}\1_A(Y^{x,n})\\
			&\times\exp\bigg(\int_{0}^{T}B(t,Y_t^{x,n})^\top\left(\sigma(t,Y^{x,n}(t))\sigma(t,Y^{x,n}(t))^\top\right)^{-1}\dd{Y^{x,n}}(t)\\
			&-\frac{1}{2}\int_{0}^{T}(2b(t,Y^{x,n}(t))+B(t,Y_t^{x,n}))^\top\left(\sigma(t,Y^{x,n}(t))\sigma(t,Y^{x,n}(t))^\top\right)^{-1}B(t,Y_t^{x,n})\dd{t}\bigg)\bigg]\\
			=&\lim\limits_{n\to\infty}\E_{\mathds{Q}^{x,n}}\bigg[\1_{\tau^n(Y^{x,n})=T}\1_A(Y^{x,n})\exp\bigg(\int_{0}^{T}\left(\sigma(t,Y^{x,n}(t))^{-1}B(t,Y^{x,n}_t)\right)^\top\dd{\tilde{W}^{x,n}}(t)\\
			&-\frac{1}{2}\int_{0}^{T}\abs{\sigma(t,Y^{x,n}(t))^{-1}B(t,Y_t^{x,n})}^2\dd{t}\bigg)\bigg]\\
			=&\mathds{Q}^x(X^x\in A)
		\end{align*}
		for all measurable $A\subset C([-r,T],\R^d)$.
	\end{proof}
	\begin{Lemma}
		\label{expfX}
		Assume conditions \eqref{singulardriftc}, \eqref{sigmac} and \eqref{delaydriftc}. Let $T>0$, $R>0$ and $p',q'\in(1,\infty)$ be given with
		\begin{align*}
			\frac{d}{p'}+\frac{2}{q'}<2.
		\end{align*}
		Then for each weak solution $(X^x,\tilde{W}^x,\mathds{Q}^x)$ of equation \eqref{eq} on $[-r,T]$ with $\norm{x}_\infty\leq R$, one has 
		$$\sup\limits_{f\in L^{q'}_{p'}(T):\norm{f}_{L^{q'}_{p'}(T)}\leq R}\E_{\mathds{Q}^x}\exp\left(\int_{0}^{T}f(t,X^x(t))\dd{t}\right)\leq C_R.$$
		with a constant $C_R=C_R(p,q,p',q',d,T,C_\sigma,\norm{b}_{L^q_p(T)},g_T)$. Additionally, one has
		$$\E_{\mathds{Q}^x}\int_{0}^{T}f(s,X^x(s))\dd{s}\leq C\norm{f}_{L^{q'}_{p'}(T)}$$
		with a constant $C>0$.
	\end{Lemma}	
	\begin{proof}
		As before, let
		$$a^x(t):=\sigma(t,M^x(t))^{-1}\left(B(t,M^x_t)+b(t,M^x(t))\right), \ t\in[0,T].$$
		By Theorem \ref{ExUni}, one has
		\begin{align*}
			&\E_{\mathds{Q}^x}\exp\left(\int_{0}^{T}f(t,X^x(t))\dd{t}\right)\\
			=&\E_{\Prob}\exp\bigg(\int_{0}^{T}f(t,M^x(t))\dd{t}+\int_{0}^{T}a^x(t)^\top\dd{W}(t)-\frac{1}{2}\int_{0}^{T}\abs{a^x(t)}^2\dd{t}\bigg)\\
			\leq&\left[\E_{\Prob}\exp\left(\int_{0}^{T}2f(t,M^x(t))\dd{t}\right)\right]^\frac{1}{2}\bigg[\E_\Prob\exp\bigg(2\int_{0}^{T}a^x(t)^\top\dd{W}(t)-\int_{0}^{T}\abs{a^x(t)}^2\dd{t}\bigg)\bigg]^\frac{1}{2}\\
			\leq&\left[\E_{\Prob}\exp\left(\int_{0}^{T}2f(t,M^x(t))\dd{t}\right)\right]^\frac{1}{2}
			\cdot\left[\E_\Prob\exp\left(6\int_{0}^{T}\abs{a^x(t)}^2\dd{t}\right)\right]^\frac{1}{4}.
		\end{align*}
		The uniform bound follows from condition \eqref{delaydriftc}, Lemma \ref{expfM} and Lemma \ref{expsupM}. By Lemma \ref{Krylov}, one has
		$$\int_{0}^{T}f(t,M^x(t))\dd{s}\to0\text{ in probability}$$
		if $\norm{f}_{L^{q'}_{p'}(T)}\to0$. Together with the exponential bound from above, it follows
		$$\E_{\mathds{Q}^x}\int_{0}^{T}f(t,X^x(t))\dd{t}\to0$$
		if $\norm{f}_{L^{q'}_{p'}(T)}\to0$. Consequently, the linear operator $A:L^{q'}_{p'}(T)\to\R$ given by
		$$f\mapsto\E_{\mathds{Q}^x}\int_{0}^{T}f(t,X^x(t))\dd{t}$$
		is continuous, which provides the existence of the desired constant.
	\end{proof}
	\begin{Lemma}
		\label{expsupX}
		Assume conditions \eqref{singulardriftc}, \eqref{sigmac} and \eqref{delaydriftc} and let $T>0$ be given. Then one has for every weak solution $(X^x,\tilde{W}^x,\mathds{Q}^x)$ of equation \eqref{eq} on $[-r,T]$ the inequality
		\begin{align*}
			&\E_{\mathds{Q}^x}\exp\left(\alpha\sup\limits_{-r\leq t\leq T}\abs{X^x(t)}^2\right)\\
			\leq&\frac{C}{\sqrt[4]{1-4\alpha dC_\sigma T}}\exp\left(\frac{\alpha}{1-4\alpha dC_\sigma T}\norm{x}_\infty^2+(16dC_\sigma T)^{-1}\norm{x}_\infty^2\right)
		\end{align*}
		for all $0\leq\alpha<(4dC_\sigma T)^{-1}$ and a constant $C=C(d,T,C_\sigma,p,q,\norm{b}_{L^q_p(T)},g_T)$.
	\end{Lemma}
	\begin{proof}
		As before, let
		$$a^x(t):=\sigma(t,M^x(t))^{-1}\left(B(t,M^x_t)+b(t,M^x(t))\right), \ t\in[0,T].$$
		By the assumed conditions, Lemmas \ref{expfM}, \ref{expsupM}, Young's inequality and H\"older's inequality, one has
		\begin{align*}
			&\E_\Prob\exp\left(6\int_{0}^{T}\abs{a^x(t)}^2\dd{t}\right)\\
			\leq&C_1\sqrt{\E_\Prob\exp\left(12 C_\sigma\int_{0}^{T}\abs{B(t,M^x_t)}\dd{t}\right)}\\
			\leq&C_2\sqrt{\E_\Prob\exp\left((4dC_\sigma T)^{-1}\sup\limits_{-r\leq r\leq T}\abs{M^x(t)}^2\right)}\\
			\leq&C_2\exp\left[(8dC_\sigma T)^{-1}\left(\norm{x}_\infty^2-\abs{x(0)}^2\right)\right]\sqrt{\E_\Prob\exp\left(\sup\limits_{0\leq t\leq T}\abs{M^x(t)}^2\right)}\\
			\leq&C_3\exp\left((4dC_\sigma T)^{-1}\norm{x}_\infty^2\right)
		\end{align*}
		for constants $C_1$, $C_2$ and $C_3$ that only depend on $d$, $T$, $C_\sigma$, $p$, $q$, $\norm{b}_{L^q_p(T)}$ and $g_T$. By Theorem \ref{ExUni}, one obtains
		\begin{align*}
			&\E_{\mathds{Q}^x}\left(\alpha\sup\limits_{-r\leq t\leq T}\abs{X^x(t)}^2\right)\\
			=&\E_\Prob\exp\left(\alpha\sup\limits_{-r\leq t\leq T}\abs{M^x(t)}^2+\int_{0}^{T}a^x(t)^\top\dd{W}(t)-\frac{1}{2}\int_{0}^{T}\abs{a^x(t)}^2\dd{t}\right)\\
			\leq&\left[\E_\Prob\exp\left(2\alpha\sup\limits_{-r\leq t\leq T}\abs{M^x(t)}^2\right)\right]^{\frac{1}{2}}\cdot\left[\E_\Prob\exp\left(2\int_{0}^{T}a^x(t)^\top\dd{W}(t)-\int_{0}^{T}\abs{a^x(t)}^2\dd{t}\right)\right]^{\frac{1}{2}}\\
			\leq&\left[\E_\Prob\exp\left(2\alpha\sup\limits_{-r\leq t\leq T}\abs{M^x(t)}^2\right)\right]^{\frac{1}{2}}\cdot\left[\E_\Prob\exp\left(6\int_{0}^{T}\abs{a^x(t)}^2\dd{t}\right)\right]^{\frac{1}{4}}\\
			\leq&\frac{C}{\sqrt[4]{1-4\alpha dC_\sigma T}}\exp\left(\frac{\alpha}{1-4\alpha dC_\sigma T}\norm{x}_\infty^2+(16dC_\sigma T)^{-1}\norm{x}_\infty^2\right)
		\end{align*}
		for a constant $C=C(d,T,C_\sigma,p,q,\norm{b}_{L^q_p(T)},g_T)$.
	\end{proof}
	\subsection{Stability}
	The following stability result has essentially been shown in \cite{Stefan1}, where $\sigma$ was supposed to be weakly differentiable with
	$$\abs{\nabla_x\sigma^{i,j}}\in L^q_{loc}\left([0,T];L^p\left(\R^d\right)\right), \ i,j=1,\dots,d.$$
	Formally, condition \eqref{sigmac} does not imply this integrability assumption and for the convenience of the reader, we provide the completely similar proof.
	\begin{Theorem}
		\label{stab}
		Assume conditions \eqref{singulardriftc}, \eqref{sigmac}, \eqref{delaydriftc} and \eqref{delaydriftlip}. Then one has for any $T_0,R>0$ and $\gamma\geq1$
		$$\E\norm{X_t^x-X_t^y}_\infty^\gamma\leq C\norm{x-y}_\infty^\gamma, \ 0\leq t\leq T_0$$
		for all $x,y\in\mathcal{C}$ with $\norm{x}_\infty,\norm{y}_\infty\leq R$ and some constant $C$ depending only on $\gamma$, $d$, $p$, $q$, $T_0$, $C_\sigma$, $\norm{b}_{L^q_p(T_0)}$, $g_T$, $C_B$ and $R$.
	\end{Theorem}
	By Theorem \ref{PDE}, for every $0<T\leq T_0$, there exists a solution
	$$\tilde{u}(\cdot;T)\in\left(H^q_{2,p}(T_0)\right)^d$$
	of the coordinatewise PDE system
	\begin{align*}
		\partial_t\tilde{u}(t,x;T)+L_t\tilde{u}(t,x;T)+b(t,x)&=0,\\
		\tilde{u}(T,x;T)&=0
	\end{align*}
	for all $t\in[0,T]$ and $x\in\R^d$ where
	$$L_tv(t,x):=\frac{1}{2}\sum_{i,j,k=1}^{d}\sigma^{i,k}(t,x)\sigma^{j,k}(t,x)\partial_i\partial_jv(t,x)+b(t,x)\cdot\nabla v(t,x), \ v\in H^q_{2,p}(T_0).$$
	Additionally, it holds
	$$\sup\limits_{T\in[0,T_0]}\norm{\tilde{u}^i(\cdot;T)}_{H^q_{2,p}(T)}<\infty, \ i=1,\dots,d$$
	and by the embedding Theorem \ref{embedding}, there exists a uniform $\delta$ such that for all $0\leq S\leq T$ with $T-S\leq\delta$
	$$\abs{\tilde{u}(t,x;T)-\tilde{u}(t,y;T)}\leq\frac{1}{2}\abs{x-y}$$
	for all $t\in[S,T]$ and $x,y\in\R^d$. Furthermore, the function
	$$u(t,x;T):=\tilde{u}(t,x;T)+x$$
	satisfies coordinatewise the equation
	\begin{align*}
		\partial_tu(t,x;T)+L_tu(t,x;T)&=0,\\
		u(T,x;T)&=x.
	\end{align*}
	\begin{proof}
		Choose $\delta>0$ like above. By induction, it suffices to prove for every $0\leq S\leq T\leq T_0$ with $T-S\leq\delta$ the implication
		\begin{align*}
			&\E\norm{X_S^x-X_S^y}_\infty^\gamma\leq C_1\norm{x-y}_\infty^\gamma \ \forall x,y\in\Co,\norm{x}_\infty,\norm{y}_\infty\leq R\\
			\implies&\E\norm{X_T^x-X_T^y}_\infty^\gamma\leq C_2\norm{x-y}_\infty^\gamma \ \forall x,y\in\Co,\norm{x}_\infty,\norm{y}_\infty\leq R
		\end{align*}
		with some constants $C_1$ and $C_2$ depending only on $\gamma$, $d$, $p$, $q$, $T_0$, $C_\sigma$, $\norm{b}_{L^q_p(T_0)}$, $g_T$, $C_B$ and $R$. For the sake of simplicity, we write $u(\cdot):=u(\cdot;T)$. Furthermore, define
		\begin{align*}
			Y^x(t)&:=u(t,X^x(t)), \ S\leq t\leq T,\\
			Y^y(t)&:=u(t,X^y(t)), \ S\leq t\leq T.
		\end{align*}
		By the choice of $\delta$, one has for the difference processes $Z(t):=X^x(t)-X^y(t)$ and $\tilde{Z}(t):=Y^x(t)-Y^y(t)$
		$$\frac{1}{2}\abs{\tilde{Z}(t)}\leq\abs{Z(t)}\leq\frac{3}{2}\abs{\tilde{Z}(t)}, \ S\leq t\leq T.$$
		Due to Lemma \ref{expfX}, Lemma \ref{Ito} is applicable, which gives
		\begin{align*}
			\tilde{Z}(t)=&\int_{S}^{t}\left(Du(s,X^x(s))B(s,X^x_s)-D u(s,X^y(s))B(s,X^y_s)\right)\dd{s}\\
			&+\int_{S}^{t}\left(Du(s,X^x(s))\sigma(s,X^x(s))-D u(s,X^y(s))\sigma(s,X^y(s))\right)\dd{W}(s)
		\end{align*}
		and consequently
		\begin{align*}
			&\dd{\abs{\tilde{Z}}^{2\gamma}}(t)\\
			=&2\gamma\abs{\tilde{Z}(t)}^{2\gamma-2}\tilde{Z}(t)^\top\left(Du(t,X^x(t))B(t,X^x_t)-Du(t,X^y(t))B(t,X^y_t)\right)\dd{t}\\
			&+2\gamma\abs{\tilde{Z}(t)}^{2\gamma-2}\tilde{Z}(t)^\top\left(Du(t,X^x(t))\sigma(t,X^x(t))-Du(t,X^y(t))\sigma(t,X^y(t))\right)\dd{W}(t)\\
			&+\gamma\abs{\tilde{Z}(t)}^{2\gamma-2}\norm{Du(t,X^x(t))\sigma(t,X^x(t))-Du(t,X^y(t))\sigma(t,X^y(t))}_{HS}^2\dd{t}\\
			&+2\gamma(\gamma-1)\abs{\tilde{Z}(t)}^{2\gamma-4}\\
			&\times\abs{\left(Du(t,X^x(t))\sigma(t,X^x(t))-Du(t,X^y(t))\sigma(t,X^y(t))\right)^\top\tilde{Z}(t)}^2\dd{t}.
		\end{align*}
		Using the boundedness of $Du$ and condition \eqref{delaydriftc} gives for $S\leq t_1\leq t_2\leq T$
		\begin{align*}
			&\abs{\tilde{Z}(t_2)}^{2\gamma}-\abs{\tilde{Z}(t_1)}^{2\gamma}\\
			\leq&c\int_{t_1}^{t_2}\norm{\tilde{Z}_s}_\infty^{2\gamma}\dd{s}\\
			+&c\int_{t_1}^{t_2}\abs{\tilde{Z}(s)}^{2\gamma-1}\norm{Du(s,X^x(s))-Du(s,X^y(s))}_{op}\abs{B(s,X^x_s)}\dd{s}\\
			+&c\int_{t_1}^{t_2}\abs{\tilde{Z}(s)}^{2\gamma-2}\tilde{Z}(s)^\top\left(Du(s,X^x(s))\sigma(s,X^x(s))-Du(s,X^y(s))\sigma(s,X^y(s))\right)\dd{W}(s)\\
			+&c\int_{t_1}^{t_2}\abs{\tilde{Z}(s)}^{2\gamma-2}\norm{Du(s,X^x(s))\sigma(s,X^x(s))-Du(s,X^y(s))\sigma(s,X^y(s))}_{HS}^2\dd{s}\\
			=&I_1+I_2+I_3+I_4
		\end{align*}
		where $c>0$ is a constant depending only on $\gamma$, $d$, $p$, $q$, $T_0$, $C_\sigma$, $\norm{b}_{L^q_p(T_0)}$, $g_T$, $C_B$ and $R$. The idea is to apply the stochastic Gronwall Lemma \ref{Gronwall}. To get rid of the badly behaving terms $I_2$ and $I_4$, one can use a suitable multiplier of the form $e^{-A(t)}$ - as in \cite{Fedrizzi2} - where $A$ is an adapted, continuous process. Here, we choose
		\begin{align*}
			A(t)&:=c\int_{S}^{t}\abs{B(s,X^x_s)}\frac{\norm{Du(s,X^x(s))-Du(s,X^y(s))}_{op}}{\abs{\tilde{Z}(s)}}\1_{\tilde{Z}(s)\neq0}\dd{s}\\
			&+c\int_{S}^{t}\frac{\norm{Du(s,X^x(s))\sigma(s,X^x(s))-Du(s,X^y(s))\sigma(s,X^y(s))}_{HS}^2}{\abs{\tilde{Z}(s)}^2}\1_{\tilde{Z}(s)\neq0}\dd{s}
		\end{align*}
		for $S\leq t\leq T$.
		To show that $A$ is indeed well defined, it suffices to show the existence of a constant $\hat{C}=\hat{C}(\gamma,d,p,q,C_\sigma,T_0,\norm{b}_{L^q_p(T_0)},g_T,C_B,R)\geq0$ such that
		$$\E \exp\left(\frac{1}{2}A(T)\right)\leq \hat{C}.$$
		Since $u$ belongs coordinatewise to $ H^q_{2,p}(T_0)$ and by condition \eqref{sigmac}, it holds
		$$(Du\cdot\sigma)^{i,j}\in L^q\left(T_0;W^{1,p}\left(\R^d\right)\right), \ i,j=1,\dots,d.$$
		Additionally, $C^\infty_c\left(\R^{d+1}\right)$ is dense in $L^q\left(T_0;W^{1,p}\left(\R^d\right)\right)$. Hence, by Young's inequality, Lemma \ref{expfX} and Lemma \ref{expsupX}, it suffices to show for all $\tilde{R}>0$ the existence of a constant $C_{\tilde{R}}=C_{\tilde{R}}(d,p,q,C_\sigma,T_0,\norm{b}_{L^q_p(T_0)},g_T,C_B,R)$ such that
		$$\E\exp\left(\int_{S}^{T}\frac{\abs{f(s,X^x(s))- f(s,X^y(s))}^2}{\abs{\tilde{Z}(s)}^2}\1_{\tilde{Z}(s)\neq0}\dd{s}\right)\leq C_{\tilde{R}}$$ 
		for all $f\in C^\infty\left(\R^{d+1}\right)$ with $\norm{f}_{L^q\left(T_0;W^{1,p}\left(\R^d\right)\right)}\leq\tilde{R}$. By Lemmas \ref{expfX} and \ref{Hardy-Littlewood}, one obtains
		\begin{align*}
			&\E\exp\left(\int_{S}^{T}\frac{\abs{f(s,X^x(s))- f(s,X^y(s))}^2}{\abs{\tilde{Z}(s)}^2}\1_{\tilde{Z}(s)\neq0}\dd{s}\right)\\
			\leq&\E\exp\left(C_d^2\int_{S}^{T}\left(\mathcal{M}\abs{\nabla f}(X^x(s))+\mathcal{M}\abs{\nabla f}(X^y(s))\right)^2\dd{s}\right)\\
			\leq&C_{\tilde{R}}
		\end{align*}
		where $C_{\tilde{R}}=C_{\tilde{R}}(d,p,q,C_\sigma,T_0,\norm{b}_{L^q_p(T_0)},g_T,C_B,R)$. By the It\=o formula, it holds
		$$e^{-A(t)}\abs{\tilde{Z}(t)}^{2\gamma}\leq\abs{\tilde{Z}(S)}^{2\gamma}+ c\int_{S}^{t}e^{-A(s)}\norm{\tilde{Z}_s}_\infty^{2\gamma}\dd{s}+\text{local martingale}.$$
		Applying the stochastic Gronwall Lemma \ref{Gronwall} gives
		$$\E\left[\sup_{S\leq t\leq T}e^{-\frac{1}{2}A(t)}\abs{\tilde{Z}(t)}^\gamma\right]\leq \tilde{C}\E\norm{\tilde{Z}_S}_\infty^\gamma\leq\tilde{C}C_1\norm{x-y}_\infty^\gamma$$
		for a constant $\tilde{C}=\tilde{C}(\gamma,d,p,q,C_\sigma,T_0,\norm{b}_{L^q_p(T_0)},g_T,C_B,R)$. Due to the estimates from above, the Cauchy-Schwarz inequality and by redefining $\gamma:=2\gamma$, one finally obtains
		\begin{align*}
			&\E\left[\sup_{S\leq t\leq T}\abs{Z(t)}^\gamma\right]\\
			\leq&\left(\E e^{\frac{1}{2}A(T)}\right)^\frac{1}{2}\left[\E\left(\sup_{S\leq t\leq T}e^{-\frac{1}{2}A(t)}\abs{Z(t)}^{2\gamma}\right)\right]^{\frac{1}{2}}\\
			\leq&C_2\norm{x-y}_\infty^\gamma
		\end{align*}
		for some constant $C_2=C_2(\gamma,d,p,q,C_\sigma,T_0,\norm{b}_{L^q_p(T_0)},g_T,C_B,R)$.
	\end{proof}
	\subsection{Strong Feller Property}
	
	The following theorem is a consequence of a log-Harnack inequality that has been shown in \cite{WangYuan} and requires the Lipschitz-continuity of $\sigma$ in space.
	\begin{Theorem}
		\label{martfeller}
		Assume condition \eqref{sigmac}. Then one has for all $t>r$
		$$\lim\limits_{y\to x}\E f(M_t^y)=\E f(M^x_t) \ \forall f\in B_b(\Co).$$
	\end{Theorem}
	Although, the drift-free equation has no delay, we need the strong Feller property with respect to the state space of path segments $\Co$. For Theorem \ref{martfeller}, we do not know so far whether one can weaken the Lipschitz condition on $\sigma$ to the ``usual'' assumption
	$$\abs{\nabla_x\sigma^{i,j}}\in L^q_{loc}\left([0,T];L^p\left(\R^d\right)\right), \ i,j=1,\dots,d.$$
	Then the following results would still hold even with that weaker assumption.
	\begin{Lemma}
		\label{bconv}
		Assume condition \eqref{sigmac}. Then one has
		$$\lim\limits_{y\to x}\E\int_{0}^{T}\abs{b(t,M^x(t))-b(t,M^y(t))}^2\dd{t}=0.$$
	\end{Lemma}
	\begin{proof}
		By Theorems \ref{goodconv} and \ref{martfeller}, one has for all $t>0$
		$$\lim\limits_{y\to x}\E\abs{f(M^x(t))-f(M^y(t))}=0 \ \forall f\in B_b(\R^d).$$
		Therefore, one has for all $f\in B_b\left([0,T]\times\R^d\right)$
		$$\lim\limits_{y\to x}\E\int_{0}^{T}\abs{f(t,M^x(t))-f(t,M^y(t))}\dd{t}=0.$$
		Consequently, $b(\cdot,M^{x_n}(\cdot))$ converges to $b(\cdot,M^x(\cdot))$ in measure with respect to $\Prob\otimes\lambda_{|[0,T]}$ for each sequence $(x_n)_{n\in\N}\subset\Co$ converging to $x$. By Lemma \ref{Krylov}, it follows
		$$\lim\limits_{\alpha\to\infty}\sup_{y\in\Co}\E\int_{0}^{T}\1_{\abs{b(t,M^y(t))}\geq\alpha}\abs{b(t,M^y(t))}^2\dd{t}=0.$$
		Hence, $\left\{\abs{b(t,M^y(t))}^2:y\in\Co\right\}$ is uniformly integrable and the stated $L^2$-convergence follows.
	\end{proof}
	\begin{proof}[Proof of Theorem \ref{StrongFeller}]
		Let $t>r$. Due to Theorems \ref{goodconv} and \ref{stab}, it suffices to show that
		$$\lim\limits_{y\to x}\E_{\mathds{Q}^y}f(X^y_t)=\E_{\mathds{Q}^x}f(X^x_t) \ \forall f\in B_b(\Co).$$
		Let $f\in B_b(\Co)$, then one has by Theorem \ref{ExUni}
		\begin{align*}
			&\E_{\mathds{Q}^x}f(X_t^x)-\E_{\mathds{Q}^y}f(X_t^y)\\
			=&\E_\Prob(D^x(t)f(M_t^x))-\E_\Prob(D^y(t)f(M_t^y))\\
			=&\E_\Prob[D^x(t)(f(M_t^x)-f(M_t^y))]+\E_\Prob[(D^x(t)-D^y(t))f(M_t^y)]\\
			\leq&\E_\Prob[D^x(t)(f(M_t^x)-f(M_t^y))]+\norm{f}_\infty\E_\Prob\abs{D^x(t)-D^y(t)}
		\end{align*}
		where we define for every $z\in\Co$
		\begin{align*}
			a^z(t)&:=\sigma(t,M^z(t))^{-1}\left(B(t,M_t^z)+b(t,M^z(t))\right),\\
			D^z(t)&:=\exp\left(\int_{0}^{t}a^z(s)^\top\dd{W}(s)-\frac{1}{2}\int_{0}^{t}\abs{a^z(s)}^2\dd{s}\right).
		\end{align*}
		By condition \eqref{sigmac}, It\=o's formula and the stochastic Gronwall Lemma \ref{Gronwall}, it holds
		$$\lim\limits_{n\to\infty}M^{x_n}_t=M^x_t\text{ in probability}$$
		for each sequence $(x_n)_{n\in\N}\subset\Co$ converging to $x$. Applying Theorems \ref{goodconv} and \ref{martfeller}, gives
		$$\lim\limits_{y\to x}\E_\Prob\abs{f(M_t^y)-f(M_t^x)}=0$$
		and in particular,
		$$\lim\limits_{n\to\infty}D^x(t)f(M_t^{x_n})=D^x(t)f(M_t^x)\text{ in probability}$$
		for each sequence $(x_n)_{n\in\N}\subset\Co$ converging to $x$. By the dominated convergence theorem, it follows
		$$\lim\limits_{y\to x}\E_\Prob[D^x(t)(f(M_t^y)-f(M_t^x))]=0.$$
		Consequently, it remains to show that
		$$\lim\limits_{y\to x}\E_\Prob\abs{D^y(t)-D^x(t)}=0.$$
		Since one has $\E_\Prob D^z(t)=1$ for all $z\in\Co$, it suffices to show for each sequence $(x_n)_{n\in\N}\subset\Co$ converging to $x$
		$$\lim\limits_{n\to\infty}D^{x_n}(t)=D^x(t)\text{ in probability}.$$
		This can be seen as follows: then by Fatou's lemma,
		$$2-\lim\limits_{n\to\infty}\E_\Prob\abs{D^{x_n}(t)-D^x(t)}=\lim\limits_{n\to\infty}\E_\Prob\left(D^x(t)+D^{x_n}(t)-\abs{D^{x_n}(t)-D^x(t)}\right)\geq2$$
		holds and the $L^1$-convergence would be an immediate consequence. Therefore, it is sufficient to show
		$$\lim\limits_{y\to x}\E_\Prob\int_{0}^{t}\abs{a^y(s)-a^x(s)}^2\dd{s}=0$$
		by the martingale isometry. One has
		\begin{align*}
			&\E_\Prob\int_{0}^{t}\abs{a^y(s)-a^x(s)}^2\dd{s}\\
			\leq&2\E_\Prob\int_{0}^{t}\norm{\sigma\left(s,M^y(s)\right)^{-1}-\sigma\left(s,M^x(s)\right)^{-1}}^2_{op}\abs{B(s,M^x_s)+b(t,M^x(s))}^2\dd{s}\\
			&+2C_\sigma\E_\Prob\int_{0}^{t}\abs{B(s,M^y_s)+b(s,M^y(s))-B(s,M^x_s)-b(s,M^x(s))}^2\dd{t}.
		\end{align*}
		The second term converges to zero by condition \eqref{delaydriftc}, Theorem \ref{goodconv}, Lemma \ref{expsupM}, Theorem \ref{martfeller} and Lemma \ref{bconv}. Moreover, for each sequence $(x_n)_{n\in\N}\subset\Co$ converging to $x$,
		$$\lim\limits_{n\to\infty}\norm{\sigma\left(\cdot,M^{x_n}(\cdot)\right)^{-1}-\sigma\left(\cdot,M^x(\cdot)\right)^{-1}}_{op}=0\text{ in measure w.r.t. }\Prob\otimes\lambda_{|[0,t]}$$
		holds by Theorem \ref{stab}, the continuity of $\sigma$ in space and the continuity of the inverting map $A\mapsto A^{-1}$ on the space of invertible matrices. Additionally, one can bound the first integrand by
		$$2C_\sigma\abs{B(\cdot,M^x_\cdot)+b(\cdot,M^x(\cdot))}^2,$$
		which is $\Prob\otimes\lambda_{|[0,t]}$-integrable by Lemma \ref{Krylov}. Consequently, one can apply the dominated convergence theorem and the proof is complete.
	\end{proof}
	\appendix
	\section{Appendix}
	\begin{Theorem}
		\label{PDE}
		Assume conditions \eqref{singulardriftc} and \eqref{sigmac}. Then for any $T>0$ and $f\in L^q_p(T)$, there exists a unique solution $u\in H^q_{2,p}(T)$ of the following PDE
		\begin{align*}
		\partial_tu(t,x)+L_tu(t,x)+f(t,x)&=0,\\
		u(T,x)&=0
		\end{align*}
		with the bound
		$$\norm{u}_{H^q_{2,p}(S,T)}\leq C\norm{f}_{L^q_p(S,T)}$$
		for any $S\in[0,T]$ and some constant $C=C(T,C_\sigma,p,q,\norm{b}_{L^q_p(T)})>0$.
	\end{Theorem}
	\begin{proof}
		See \cite{Zhang2011}.
	\end{proof}
	\begin{Theorem}
		\label{embedding}
		Let $p,q\in(1,\infty)$, $T>0$ and $u\in H^q_{2,p}(T)$.
		\begin{enumerate}
			\item If $\frac{d}{p}+\frac{2}{q}<2$, then $u$ is a bounded H\"older continuous function on $[0,T]\times\R^d$ and for any $0<\varepsilon$, $\delta\leq1$ satisfying
			$$\varepsilon+\frac{d}{p}+\frac{2}{q}<2, \ \ 2\delta+\frac{d}{p}+\frac{2}{q}<2,$$
			there exists a constant $N=N(p,q,\varepsilon,\delta)$ such that
			\begin{align*}
			\abs{u(t,x)-u(s,x)}&\leq N\abs{t-s}^\delta\norm{u}^{1-\frac{1}{q}-\delta}_{\mathds{H}^q_{2,p}(T)}\norm{\partial_tu}^{\frac{1}{q}+\delta}_{L^q_p(T)},\\
			\abs{u(t,x)}+\frac{\abs{u(t,x)-u(t,y)}}{\abs{x-y}^\varepsilon}&\leq NT^{-\frac{1}{q}}\left(\norm{u}_{\mathds{H}^q_{2,p}(T)}+T\norm{\partial_tu}_{L^q_p(T)}\right)
			\end{align*}
			for all $s,t\in[0,T]$ and $x,y\in\R^d,x\neq y$.
			\item If $\frac{d}{p}+\frac{2}{q}<1$, then $\nabla u$ is a bounded H\"older continuous function on $[0,T]\times\R^d$ and for any $\varepsilon\in(0,1)$ satisfying
			$$\varepsilon+\frac{d}{p}+\frac{2}{q}<1,$$
			there exists a constant $N=N(p,q,\varepsilon)$ such that
			\begin{align*}
			\abs{\nabla u(t,x)-\nabla u(s,x)}&\leq N\abs{t-s}^\delta\norm{ u}^{1-\frac{1}{q}-\frac{\varepsilon}{2}}_{\mathds{H}^q_{2,p}(T)}\norm{\partial_tu}^{\frac{1}{q}+\frac{\varepsilon}{2}}_{L^q_p(T)},\\
			\abs{\nabla u(t,x)}+\frac{\abs{\nabla u(t,x)-\nabla u(t,y)}}{\abs{x-y}^\varepsilon}&\leq NT^{-\frac{1}{q}}\left(\norm{u}_{\mathds{H}^q_{2,p}(T)}+T\norm{\partial_tu}_{L^q_p(T)}\right)
			\end{align*}
			for all $s,t\in[0,T]$ and $x,y\in\R^d,x\neq y$.
		\end{enumerate}
	\end{Theorem}
	\begin{proof}
		See \cite[p. 22, 23, 36]{Fedrizzi}.
	\end{proof}
	In the next lemma we identify every $u\in H^q_{2,p}$ with its regular version.
	\begin{Lemma}[It\=o formula for $H^q_{2,p}$-functions]
		\label{Ito}
		Let $T>0$, $p>1$ and $q>1$ satisfying \eqref{pq}. Let $X:\Omega\times[0,T]\to\R^d$ be a semimartingale on some filtrated probability space $\left(\Omega,\mathcal{F},\Prob,\left(\mathcal{F}_t\right)_{t\geq0}\right)$ of the form
		$$\dd{X}(t)=b(t)\dd{t}+\sigma(t)\dd{W}(t)$$
		where $W$ is a $d$-dimensional Brownian motion, $b:\Omega\times[0,T]\to\R^d$ and $\sigma:\Omega\times[0,T]\to\R^{d\times d}$ are progressively measurable with
		$$\Prob\left(\norm{b}_{L^1[0,T]}+\norm{a^{i,j}}_{L^\delta[0,T]}<\infty\right)=1, \ i,j=1,\dots,d$$
		for some $1<\delta\leq\infty$ where $a:=\sigma\sigma^\top$. Furthermore, assume that there exists a constant $C>0$ with
		$$\E\int_{0}^{T}f(t,X(t))\dd{t}\leq C\norm{f}_{L^{q/\delta*}_{p/\delta^*}(T)}$$
		for all $f\in L^{q/\delta^*}_{p/\delta^*}(T)$ where $\delta^*$ denotes the conjugate exponent of $\delta$.
		Then for any $u\in H^q_{2,p}(T)$, the It\=o formula holds, i.e.
		\begin{align*}
			u(t,X(t))-u(0,X(0))=&\int_{0}^{t}\partial_tu(s,X(s))\dd{s}+\int_{0}^{t}\nabla u(s,X(s))^\top b(s)\dd{s}\\
			&+\int_{0}^{t}\nabla u(s,X(s))^\top\sigma(s)\dd{W}(s)\\
			&+\frac{1}{2}\sum_{i,j=1}^{d}\int_{0}^{t}\partial_i\partial_j u(s,X(s))a^{i,j}(s)\dd{s}.
		\end{align*}
	\end{Lemma}
	\begin{proof}
		See \cite{Stefan1}.
	\end{proof}
	Let $\phi$ be a locally integrable function on $\R^d$. The Hardy-Littlewood maximal function is defined by
	$$\mathcal{M}\phi(x):=\sup_{0<r<\infty}\frac{1}{\abs{B_r}}\int_{B_r}\phi(x+y)\dd{y}$$
	where $B_r$ is the Euclidean ball of radius $r$. The following result is cited from \cite{Zhang2011}.
	\begin{Lemma} \ 
		\label{Hardy-Littlewood}
		\begin{enumerate}
			\item	There exists a constant $C_d>0$ such that for all $\phi\in C^{\infty}\left(\R^d\right)$ and $x,y\in\R^d$,
			$$\abs{\phi(x)-\phi(y)}\leq C_d\abs{x-y}\left(\mathcal{M}\abs{\nabla\phi}(x)+\mathcal{M}\abs{\nabla\phi}(y)\right).$$
			\item For any $p>1$, there exists a constant $C_{d,p}$ such that for all $\phi\in L^p\left(\R^d\right)$,
			$$\norm{\mathcal{M}\phi}_{L^p}\leq C_{d,p}\norm{\phi}_{L^p.}$$
		\end{enumerate}
	\end{Lemma}
	\begin{Lemma}
		\label{Gronwall}
		Let $Z$ be an adapted non-negative stochastic process with continuous paths defined on $[0,\infty)$ that satisfies the inequality
		$$Z(t)\leq K\int_{0}^{t}\sup\limits_{0\leq r\leq s}Z(r)\dd{s}+M(t)+C,$$
		where $C\geq0$, $K\geq0$ and $M$ is a continuous local martingale with $M(0)=0$. Then for each $0<p<1$, there exist universal finite constants $c_1(p)$, $c_2(p)$ (not depending on $K$, $C$, $T$ and $M$) such that
		$$\E\left[\sup\limits_{0\leq t\leq T}Z(t)^p\right]\leq C^p c_2(p)e^{c_1(p)KT}\text{ for every }T\geq0.$$
	\end{Lemma}
	\begin{proof}
		See \cite{Renesse}.
	\end{proof}
	\section{The Strict Topology on $\mathbf{C_b(E)}$}
	\label{stricttop}
	In this subsection we want to consider the strict topology for the function space $C_b(E)$ of bounded, continuous functions on a polish space $E$ as a nontrivial example where the results given before are applicable.
	
	In this subsection, we assume that $E$ is a polish space. Equipping $C_b(E)$ with the usual supremum norm might have some drawbacks if $E$ is not (locally) compact. A well-known result is the following
	\begin{Proposition}
		$(C_b(E),\norm{\cdot}_\infty)$ is separable iff $E$ is compact.
	\end{Proposition}
	Additionally, if $E$ is not locally compact, the dual space of $(C_b(X),\norm{\cdot})$ may not coincide with the space of complex Borel measures on $E$. That might give rise to consider different topologies on $C_b(E)$. The strict topology on $C_b(E)$ is defined as follows.
	\begin{Definition}
		Define
		$$H^+(E):=\left\{u:E\to\R_{\geq0}:\{u\geq\alpha\}\text{ compact for all }\alpha>0\right\}.$$
		Then the strict topology $\beta$ on $C_b(E)$ shall be generated by the seminorms
		$$\rho_u(f):=\norm{uf}_\infty, \ f\in C_b(E).$$
	\end{Definition}
	\begin{Remark}
		It turns out that a sequence converges with respect to the strict topology iff the sequence converges uniformly on compact sets and is uniformly bounded (both with respect to the supremum norm).
		
		The strict topology has a rich structure and is discussed deeply in the context of Markov processes and Feller semigroups in \cite{van2011markov}. A remarkable property is the following: the topology $\beta$ is the finest locally convex topology such that the dual space coincides with the space of complex Borel measures on $E$.
	\end{Remark}
	\begin{Definition}
		A collection $\mathcal{N}$ of subsets of $E$ is called a network if for any $x\in E$ and open $O\subseteq E$ with $x\in O$ there exists some $N\in\mathcal{N}$ with $x\in N\subseteq O$.
	\end{Definition}
	\begin{Lemma}
		The space $(C(E),\beta)$ has a countable network. In particular, it is hereditarily Lindel\"of.
	\end{Lemma}
	\begin{proof}
		The approach is analogous to the one for the topology of pointwise convergence, which can be found in Lemma 7.1 in \cite{GroupValued}. Let $\mathfrak{B}$ be a countable base for $E$. Then define for $B\in\mathfrak{B}$, $m\in\N$, $a,b\in\Q$ with $a<b$
		$$[B,a,b,m]:=\left\{h\in C_b(E):h(B)\subseteq(a,b),\norm{h}_\infty<m\right\}$$
		and
		$$\tilde{\mathcal{N}}:=\left\{[B,a,b,m]:B\in\mathfrak{B},a,b\in\Q,a<b,m\in\N\right\}.$$
		Now,
		$$\mathcal{N}:=\left\{\bigcap_{i=1}^nA_i:A_1,\dots,A_n\in\tilde{\mathcal{N}},n\in\N\right\}$$
		is countable and the candidate for the claimed network. Let $f\in C_b(E)$, $u\in H^+(E)$ and $\varepsilon>0$ be given. Then one has to show that there exists some $N\in\mathcal{N}$ such that
		$$f\in N\subseteq\left\{g\in C_b(E):\norm{u(g-f)}_\infty<\varepsilon\right\}.$$
		The first step is to use the compactness property of $u$: define
		$$m:=\inf\left\{n\in\N:n\geq\norm{f}_\infty\right\}$$
		and
		$$K:=\left\{x\in E:\abs{u(x)}\geq\frac{\varepsilon}{2m}\right\}.$$
		By construction, it holds for each $B\in\mathfrak{B}$, $a,b\in\Q$ with $a<b$
		$$\abs{u(x)(h(x)-f(x))}<\varepsilon \ \forall x\in E\setminus K, h\in[B,(a,b),m].$$
		If $K=\emptyset$, the proof would be finished. Thus, assume $K\neq\emptyset$ and define
		$$\tilde{\varepsilon}:=\frac{\varepsilon}{\norm{u}_\infty}.$$
		So, it remains to show that there exists an $N\in\mathcal{N}$ with $f\in N$ and
		$$\abs{h(x)-f(x)}<\tilde{\varepsilon} \ \forall x\in K,h\in N.$$
		Since $f$ is continuous, there exists for every $x\in K$ a $B_x\in\mathfrak{B}$ with
		$$f(B_x)\subseteq\left(f(x)-\tilde{\varepsilon}/4,f(x)+\tilde{\varepsilon}/4\right).$$
		Additionally, $K$ is compact. Consequently, there exist $x_1,\dots,x_n$, $n\in\N$ with
		$$K\subseteq\bigcup_{i=1}^nB_{x_i}.$$
		Now, choose suitable $a_1,b_1,\dots,a_n,b_n$ with
		$$f(x_i)<f(x_i)-\tilde{\varepsilon}/2<a_i<f(x_i)-\tilde{\varepsilon}/4<f(x_i)+\tilde{\varepsilon}/4<b_i<f(x_i)+\tilde{\varepsilon}/2, \ i=1\dots n.$$
		Then it holds $f\in\bigcap_{i=1}^n[B_{x_i},a_i,b_i,m]\in\mathcal{N}$ and
		$$\abs{h(x)-f(x)}<\tilde{\varepsilon} \ \forall x\in K,h\in\bigcap_{i=1}^n[B_{x_i},a_i,b_i,m].$$
	\end{proof} In the following corollary, the space $C_b(E)$ will be implicitly equipped with the strict topology $\beta$.
	\begin{Corollary}
		Let $\left(\Omega,\mathcal{F},\Prob\right)$ be some probability space and $X,X_n:\Omega\to C_b(E)$, $n\in\N$ be measurable maps. Then the following statements are equivalent
		\begin{enumerate}
			\item
			\begin{enumerate}
				\item$\lim\limits_{n\to\infty}\Prob\left(\rho_u(X-X_n)>\varepsilon\right)=0 \ \forall \varepsilon>0,u\in H^+(E)$,
				\item$\lim\limits_{n\to\infty}\Prob_{X_n}\left(O\right)=\Prob_X\left(O\right),$ for all open $O\subset C_b(E)$.
			\end{enumerate}
			\item $\lim\limits_{n\to\infty}\E\abs{f(X)-f(X_n)}=0 \ \forall f\in B_b(E).$
		\end{enumerate}
	\end{Corollary}
	\begin{Remark}
		The strict topology is not metrizable if $E$ is not compact. This can be seen as follows: assume that $(C(E),\beta)$ is metrizable and $E$ is not compact. Then the zero function has a countable neighborhood base and there is a sequence $(x_n)_{n\in\N}$ in $E$ that has no cluster points with $x_n\neq x_m$ if $n\neq m$. Thus, there exists a countable set $(u_i)_{i\in\N}\subset H^+(E)$ such that
		$$\forall u\in H^+(E) \ \exists i\in\N:u(x_n)\leq u_i(x_n) \ \forall n\in\N.$$
		Additionally, one has for every $u\in H^+(E)$
		$$\lim\limits_{n\to\infty}u(x_n)=0.$$
		Now, one can construct a $\hat{u}\in H^+(E)$ that contradicts to the inequality above. Define inductively
		$$n_1:=1, \ n_{i+1}:=\inf\left\{n\in\N:n>n_i,u_{i+1}(x_n)\leq 1/(i+1)\right\}$$
		and
		\begin{align*}
		\hat{u}\left(x_{n_i}\right)&:=u_i\left(x_{n_i}\right)+1/i,\\
		\hat{u}(y)&:=0, \ y\in E\text{ with }\nexists \ i\in\N: \ y=x_{n_i}.
		\end{align*}
		Indeed, the function $\hat{u}$ is well defined and it holds $\hat{u}\in H^+(E)$ with
		$$\hat{u}\left(x_{n_i}\right)>u_i\left(x_{n_i}\right), \ i\in\N.$$
	\end{Remark}
	\bibliographystyle{plain}
	\bibliography{Bibliography}
\end{document}